\documentclass[a4paper,doublespace]{amsart}

\usepackage[colorlinks, linkcolor=red, anchorcolor=blue, citecolor=green]{hyperref}
\usepackage{graphics,epic}
\usepackage{amsmath,amssymb, amsthm}
\usepackage[all,2cell]{xy}
\usepackage{tikz}
\usetikzlibrary{arrows}
\usepackage{shorttoc}
\usepackage[pagewise]{lineno}%\linenumbers

\newtheorem{theorem}{Theorem}[section]
\newtheorem*{theorem*}{Theorem}

\newtheorem{lemma}[theorem]{Lemma}
\newtheorem{proposition}[theorem]{Proposition}
\newtheorem{corollary}[theorem]{Corollary}

\newtheorem*{conjecture*}{Conjecture}

\newtheorem{remark}[theorem]{Remark}
\newtheorem{definition}[theorem]{Definition}

\newcommand{\ie}{{\em i.e.}\ }

\newcommand{\opname}[1]{\operatorname{\mathsf{#1}}}

\renewcommand{\mod}{\opname{mod}\nolimits}

\newcommand{\proj}{\opname{proj}\nolimits}

\newcommand{\per}{\opname{per}\nolimits}
\newcommand{\add}{\opname{add}\nolimits}

\newcommand{\der}{\cd}

\newcommand{\dimv}{\underline{\dim}\,}

\newcommand{\rank}{\opname{rank}\nolimits}
\newcommand{\rankv}{\underline{\rank}\,}

\newcommand{\ind}{\opname{ind}}

\newcommand{\lcm}{\opname{lcm}}

\newcommand{\Z}{\mathbb{Z}}

\newcommand{\Q}{\mathbb{Q}}

\renewcommand{\P}{\mathbb{P}}

\newcommand{\ra}{\rightarrow}

\newcommand{\Hom}{\opname{Hom}}
\newcommand{\go}{\opname{G_0}}

\newcommand{\Ext}{\opname{Ext}}

\newcommand{\End}{\opname{End}}

\newcommand{\ca}{{\mathcal A}}

\newcommand{\cc}{{\mathcal C}}
\newcommand{\cd}{{\mathcal D}}

\newcommand{\ct}{{\mathcal T}}

\newcommand{\cw}{{\mathcal W}}

\begin{document}

\title[On cluster algebra of type $\mathrm{C}$]{Cluster algebras arising from cluster tubes I: integer vectors}\thanks{Partially  supported by the National Natural Science Foundation of China (No. 11471224)}

\author{Changjian Fu}
\address{Changjian Fu\\Department of Mathematics\\SiChuan University\\610064 Chengdu\\P.R.China}
\email{changjianfu@scu.edu.cn}
\date{\today}
\author{Shengfei Geng}
\address{Shengfei Geng\\Department of Mathematics\\SiChuan University\\610064 Chengdu\\P.R.China}
\email{genshengfei@scu.edu.cn}
\author{Pin Liu}
\address{Pin Liu\\Department of Mathematics\\
   Southwest Jiaotong University\\
  610031 Chengdu \\
  P.R.China}
  \email{pinliu@swjtu.edu.cn}
\maketitle

\begin{center}
\textsf{\it To our teacher Liangang Peng on the occasion of his 60th birthday}
\end{center}

\begin{abstract}
We study cluster algebras arising from cluster tubes. We obtain categorical interpretations for $g$-vectors, $c$-vectors and denominator vectors for cluster algebras of type $\mathrm{C}$ with respect to arbitrary initial seeds. In particular, a denominator theorem has been proved, which enables us to establish the linearly independence of denominator vectors of cluster variables  from the same cluster for cluster algebras of type $\mathrm{A}\mathrm{B}\mathrm{C}$. This strengthens the link between cluster tubes and cluster algebras of type $\mathrm{C}$ initiated by Buan, Marsh and Vatne.
\end{abstract}

%\tableofcontents

\shorttableofcontents{}{1}

\section{Introduction and main results}~\label{s:introduction}
\subsection{Motivation}
At the beginning of this century, Fomin and Zelevinsky~\cite{FominZelevinsky02} invented a new class of algebras called cluster algebras motivated by total positivity in algebraic groups and canonical bases in quantum groups. Since their introduction, cluster algebras have found application in a diverse variety of settings which include Poisson geometry, Teichm{\"u}ller theory, tropical geometry, algebraic combinatorics and last not least the representation theory of quivers and finite-dimensional algebras. See, for example, \cite{FominZelevinsky04, Keller11} and references therein. Cluster algebras are commutative algebras endowed with a distinguished set of generators called the {\it cluster variables}. These generators are gathered into overlapping sets of fixed finite cardinality, called {\it clusters}, which are defined recursively from an initial one via an operation called {\it mutation}.  A cluster algebra is of {\it finite type} if it has only a finite number of clusters.  Fomin and Zelevinsky \cite{FominZelevinsky03} proved that the cluster algebras of finite type are parametrized by the finite root systems. In particular, cluster algebras of type $\mathrm{A}, \mathrm{D}, \mathrm{E}$ are skew-symmetric and the ones of type $\mathrm{B}, \mathrm{C}, \mathrm{F}, \mathrm{G}$ are no longer skew-symmetric but only skew-symmetrizable. 

 It was recognized in \cite{MRZ03} that the combinatorics  of cluster mutation are closely related to those of tilting theory in the representation theory of quivers. This discovery was the main motivation for the invention of cluster categories and the study of more general 2-Calabi-Yau triangulated categories ({\it cf.}~\cite{BMRRT, KR}). In the categorical setting, the cluster tilting objects play the role of the clusters, and their indecomposable direct summands the one of the cluster variables. An explicit map from the set of indecomposable factors of cluster tilting objects to the set of cluster variables was defined by Caldero and Chapoton in \cite{CalderoChapoton} for cluster categories and by Palu~\cite{Palu} for 2-Calabi-Yau triangulated categories with cluster tilting objects.
This leads to a great development for cluster algebras of skew-symmetric type by additive categorifications given by 2-Calabi-Yau triangulated categories with cluster-tilting objects. However, it is still an open question that how to construct suitable additive categorification for a general cluster algebra of skew-symmetrizable type. We refer to ~\cite{Zhu07, Demonet, GLS17} for some progress on skew-symmetrizable cluster algebras with an acyclic initial seed via categorications~({\it cf.} also~\cite{Demonet10}). 

In~\cite{BuanMarshVatne}, Buan {\it et al.} proposed that the cluster categories of tubes~(called {\it cluster tubes}) are good candidates for the combinatorics of cluster algebras of type $\mathrm{C}$ \footnote{In fact, Buan {\it et al.}~\cite{BuanMarshVatne} have considered the cluster algebra of type $\mathrm{B}$ but not of type $\mathrm{C}$. However, the cluster combinatorics of cluster algebras of type $\mathrm{B}$ is the same as  that of cluster algebras of type $\mathrm{C}$~\cite{FominZelevinsky03b}. Hence here and after, we may state the result of ~\cite{BuanMarshVatne} in cluster algebras of type $\mathrm{C}$.}.  We remark that the cluster tubes have no cluster-tilting objects but only maximal rigid objects. Nevertheless, they proved that there is a bijection between the indecomposable rigid objects of cluster tubes and the cluster variables of cluster algebras of type $\mathrm{C}$. Moreover, the bijection induces a bijection between the basic maximal rigid objects and the clusters, which is compatible with mutations. For a special choice of acyclic initial seed, an analogue of Caldero-Chapoton formula has been established in~\cite{ZhouZhu14}. The present paper and its sequel \cite{FuGengLiu} are devoted to investigate the link between cluster tubes and cluster algebras of type $\mathrm{C}$ in a full generality. Namely, we consider cluster algebra of type $\mathrm{C}$ with respect  to an arbitrary initial seed. As shown in \cite{FominZelevinsky07}, the structure of cluster algebras is to a large extent controlled by certain integer vectors called $g$-vectors, $c$-vectors and denominator vectors. In the present paper, we focus on these integer vectors arising from cluster algebras.
 We find interpretations of denominator vectors, $g$-vectors and $c$-vectors in terms of representations of algebras arising from cluster tubes.  Using these interpretations, we confirm certain conjectures about denominator vectors made by Fomin and Zelevinsky in \cite{FominZelevinsky07}. 

\subsection{Main results}
Now we describe the main results of the paper in more detail.
Fix a positive integer $n$. Let $\Delta_{n+1}$ be the cyclic quiver with $n+1$ vertices. We label the vertex set by $\{1,2,\ldots, n+1\}$ such that the arrows are precisely from vertex $i$ to $i+1$ (taken modulo $n+1$). Denote by $\ct:=\ct_{n+1}$ the category of finite-dimensional nilpotent representation over the opposite quiver $\Delta_{n+1}^{\operatorname{op}}$.  The category $\ct$ is called a {\it tube of rank $n+1$}. It is a hereditary abelian category. Each indecomposable object of $\ct$ is uniquely determined by its socle and its quasi-length. For $1\leq a\leq n+1$ and $b\in \Z$, we will denote by $(a,b)$ the unique indecomposable object with socle the simple at vertex $a$ and of quasi-length $b$. Throughout, we will use the convention that $(a, b)=0$ if $b\leq 0$ and when we write equations and inequalities which involve first coordinates outside the domain $1,\ldots, n+1$, we will implicitly assume identification modulo $n+1$.

 Let $\der^b(\ct)$ be the bounded derived category of $\ct$ with suspension functor $\Sigma$. Let $\tau$ be the Auslander-Reiten translation of $\der^b(\ct)$, where $\tau(a, b)=(a-1, b)$. The {\it cluster tube of rank $n+1$} is the orbit category $\cc:=\cc_{n+1}=\der^b(\ct)/ \tau^{-1}\circ \Sigma$. The category $\cc$ admits a canonical triangle structure such that the canonical projection $\pi:\der^b(\ct)\to \cc$ is a triangle functor \cite{Keller05} ({\it cf.} also \cite{Barot-Kussin-Lenzing}). Moreover, it is a Calabi-Yau triangulated category with Calabi-Yau dimension of $2$. The composition of the embedding of $\ct$ into $\der^b(\ct)$ with the canonical projection $\pi$ yields a bijection between the indecomposable objects of $\ct$ and the indecomposable objects of $\cc$. We always identify the objects of  $\ct$ with the ones of $\cc$ by the bijection. In particular, we may say the length of an indecomposable object of $\cc$.

 An object $T$ of $\cc$ is {\it rigid} if $\Ext^1_{\cc}(T, T)=0$. It is
{\it maximal rigid} if it is rigid and $\Ext^1_{\cc}(X\oplus T, X\oplus
T)=0$ implies that $X\in \add T$, where $\add T$ denotes the subcategory of $\cc$ consisting of objects which are finite direct sum of direct summands of $T$.   Let $T=\bigoplus_{i=1}^nT_i$ be a basic maximal rigid object of $\cc$ with the endomorphism algebra $\Gamma=\End_{\cc}(T)$ and $\mod \Gamma$ the category of finitely generated right $\Gamma$-modules.  To $T$ we associate a quiver $Q_T$ whose vertices correspond to the indecomposable direct summands of $T$ and the arrows from the indecomposable direct summand $T_i$ to $T_j$ is given by the dimension of the space of irreducible maps $\opname{rad}(T_i,T_j)/\opname{rad}^2(T_i,T_j)$, where $\opname{rad}(-,-)$ is the radical of the category $\add T$. It is clear that the quiver $Q_T$ coincides with the Gabriel quiver of the endomorphism algebra $\Gamma$.  We will study the following classical problem:
\begin{itemize}
\item Which representations are determined by their dimension vectors?
\end{itemize}
Recall that an $\Gamma$-module $M$ is {\it $\tau$-rigid} if $\Hom_\Gamma(M, \tau M)=0$, where $\tau$ is the Auslander-Reiten translation of $\mod \Gamma$. An indecomposable $\tau$-rigid $\Gamma$-module is locally free in the sense of~\cite{GLS14}, and then each indecomposable $\tau$-rigid $\Gamma$-module $M$ admits an integer vector called the {\it rank vector} $\rankv M$~({\it cf.} Section~\ref{ss:rank-vector} for the precise definition). Rank vectors are closely related to dimension vectors. Our first main result in this paper shows that indecomposable $\tau$-rigid modules are determined by their rank vectors({\it cf.}~Theorem~\ref{t:rank-vector-of-tau-rigid}).
\begin{theorem}~\label{t:thm1}
Let $\cc$ be the cluster tube of rank $n+1$ and $T=\bigoplus_{i=1}^nT_i$ a basic maximal rigid object of $\cc$. Let $\Gamma=\End_{\cc}(T)$ be the endomorphism algebra of $T$. Then different indecomposable $\tau$-rigid $\Gamma$-modules have different rank vectors. In particular, different indecomposable $\tau$-rigid $\Gamma$-modules have different dimension vectors.
\end{theorem}

We then develop some applications to the theory of cluster algebras. We refer to Section~\ref{ss: cluster algebra} for the required background for cluster algebras.
In~\cite{FominZelevinsky07}, Fomin and Zelevinsky introduced a family of combinatorial parametrization of cluster monomials by integer vectors: denominator vectors.
The parametrization by denominator vectors was independent of the choice of the coefficient system. One of the important properties of cluster algebras is the so-called {\it Laurent phenomenon}~\cite{FominZelevinsky02}: each cluster variable can be expressed as a Laurent polynomial in the initial cluster variables~$x_1,\cdots, x_n$. As a consequence of the Laurent phenomenon, for each cluster monomial $x$, there exists a unique polynomial $f(x_1,\cdots, x_n)$ which is not divisible by any $x_i$ such that 
\[x=\frac{f(x_1,\cdots, x_n)}{x_1^{d_1}\cdots x_n^{d_n}}.
\]
The {\it denominator vector} of $x$ is defined to be 
\[\opname{den}(x)=(d_1,\cdots,d_n)^{\opname{tr}}\in \Z^n,
\] 
where $(d_1,\cdots,d_n)^{\opname{tr}}$ is the transpose of $(d_1,\cdots, d_n)$.
Inspired by Lusztig's parameterization of canonical bases in quantum groups, Fomin and Zelevinsky proposed the following denominator conjecture ({\it cf.} Conjecture 7.6 of ~\cite{FominZelevinsky07}).
\begin{conjecture*}
Different cluster monomials have different denominator vectors. In particular, the denominator vectors of cluster variables in a cluster form a basis of $\mathbb{Q}^n$.
\end{conjecture*}
To our best knowledge, the denominator conjecture is still open widely. It has not yet been checked for cluster algebras of finite type. The linearly independence of denominator vectors is only known for certain cluster algebras with respect to acyclic initial seeds. More precisely, it has been verified by Caldero and Keller~\cite{CalderoKeller} for acyclic cluster algebras associated to quivers~({\it cf.} also~\cite{RupelStella}), by Sherman and Zelevinsky~\cite{SZ} for cluster algebras of rank $2$ and by Fomin and Zelevinsky~\cite{FominZelevinsky07} for cluster algebra of finite type with bipartite initial seeds. In this paper, we confirm the conjecture as follows ({\it cf.}~Theorem~\ref{t:linearly-independence}, Remark~\ref{r:linearly-independence-A} and Remark~\ref{r:linearly-independence-B}).
\begin{theorem}~\label{t:thm2}
Let $\mathcal{A}$ be a cluster algebra of type $\mathrm{A}_n, \mathrm{B}_n$ or $\mathrm{C}_n$ and $\mathrm{y}=\{y_1,\cdots,y_n\}$ an arbitrary cluster of $\mathcal{A}$. Then the denominator vectors $\opname{den}(y_1),\cdots, \opname{den}(y_n)$ are linearly independent over $\mathbb{Q}$.
\end{theorem}

As mentioned already, our proof of Theorem \ref{t:thm2} is based on interpreting denominator vectors in terms of representations of the endomorphism algebra $\Gamma$ corresponding to a basic maximal rigid object $T$ in the cluster tube $\cc$. We now describe this interpretation more precisely. Recall that an object $M$ is {\it finitely presented} by $T$ if there is a triangle $T_1^M\to T_0^M\to M\to \Sigma T_1^M$ with $T_1^M, T_0^M\in \add T$. Denote by $\opname{pr} T$ the subcategory of $\cc$ consisting of objects which are finitely presented by $T$. There is an equivalence 
 \[F:=\Hom_\cc(T,?): \opname{pr} T/\add \Sigma T\xrightarrow{\sim} \mod \Gamma,
 \]
 where $\opname{pr}T/\add \Sigma T$ is the additive quotient of $\opname{pr} T$ by morphisms factorizing through $\add \Sigma T$.  Moreover, for an indecomposable rigid object $M\not\in \add \Sigma T$, $F(M)$ is an indecomposable $\tau$-rigid $A$-module~({\it cf.}~Section~\ref{s:endomorphism-algebra}). 
 
 As in \cite{BuanMarshVatne}, we associate a  matrix $B_T=(b_{ij})\in M_n(\Z)$ to each basic maximal rigid object $T=\bigoplus_{i=1}^nT_i\in \cc$~(see Section~\ref{ss:cluster-tubes}). The matrix $B_T$ can also be constructed from the quiver $Q_T$~({\it cf.}~Lemma~\ref{l:skew-symmetrizable-matrix-via-quiver}).
 The main result of ~\cite{BuanMarshVatne} shows that the cluster tube $\cc$ admits a cluster structure for cluster algebras of type $\mathrm{C}_n$. In particular, $B_T$ is a skew-symmetrizable matrix of type $\mathrm{C}_n$ and there is a bijection from the set of indecomposable rigid objects of $\cc$ to the set of cluster variables of a cluster algebra of type $\mathrm{C}_n$. As in customary these days, we work with the $n$-regular tree $\mathbb{T}_n$ whose edges are labeled by the numbers $1,2,\cdots, n$, so that the $n$ edges emanating from each vertex carry different labels. We fix a vertex $t_0\in \mathbb{T}_n$. We refer to a family of independent variables $\mathrm{x}_0=\{x_1,\cdots, x_n\}$ as the initial cluster and $B_T$ as the exchange matrix at $t_0$. In other words, we refer to $(\mathrm{x}_0, B_T)$ as the {\it initial seed} associated to the vertex $t_0\in \mathbb{T}_n$, so that a cluster pattern is given. Denote by $\mathcal{A}_{T}$ ({\it resp.} $\mathcal{A}_{T,pr}$) the cluster algebra without coefficients ({\it resp.} with principal coefficients) associated with this cluster pattern ({\it cf. }~Section~\ref{ss: cluster algebra} for details). 
 
 Denote by $\mathbb{X}_?$ the bijection given in~\cite{BuanMarshVatne} such that $\Sigma T$ corresponds to the initial cluster of the cluster algebra $\mathcal{A}_T$.
 For each non-initial cluster variable $x$ of $\mathcal{A}_T$, there is a unique indecomposable rigid object $M\not\in \add \Sigma T$ such that $\mathbb{X}_M=x$.
We obtain a denominator theorem for the cluster algebra $\mathcal{A}_T$~({\it cf.} Theorem~\ref{t:denominator-thm}).
\begin{theorem}~\label{t:thm3}
 For each indecomposable rigid object $M\not\in \add \Sigma T$, we have
 \[\opname{den}(\mathbb{X}_M)=\opname{\underline{rank}} F(M).
 \]
 \end{theorem}

Now we turn to the cluster algebra $\mathcal{A}_{T,pr}$.
For cluster algebras with principal coefficients,
two other kinds of integer vectors have played important role in the structure theory: the $g$-vectors parametrizing cluster variables and the $c$-vectors parametrizing the coefficients. 
It is interesting to find categorical interpretations of $g$-vectors and $c$-vectors for cluster algebras which admit categorifications.
  The following result provides an interpretation of $c$-vectors of $\mathcal{A}_{T,pr}$ in terms of representations of algebras arising from cluster tubes~({\it cf.}~Theorem~\ref{t:positive-c-vectors} and  for the interpretation of $g$-vectors, we refer to Theorem~\ref{t:g-vector and index}).
\begin{theorem}~\label{t:thm4}
The positive $c$-vectors of $\mathcal{A}_{T,pr}$ are precisely the rank vectors of indecomposable $\tau$-rigid $\Gamma$-modules.
\end{theorem}
We remark that the analogue statement of Theorem~\ref{t:thm2} has been proved by N\'{a}jera Ch\'{a}vez for acyclic skew-symmetric cluster algebras~\cite{Chavez13} and skew-symmetric cluster algebras of finite type~\cite{Chavez14}~({\it cf.} also~\cite{Fu17}).

The paper is organized as follows. Section~\ref{l:preliminaries} provides the required background from cluster algebras, $\tau$-tilting theory and cluster tubes. In Section~\ref{ss:exchange-compatible}, we establish the exchange compatibility (Theorem~\ref{t:exchange compatible}) for the cluster tube $\cc$. As applications of the exchange compatibility, we prove Theorem~\ref{t:thm1} in Section~\ref{ss:rank-vector}
and the denominator theorem (Theorem~\ref{t:thm3}) in Section~\ref{ss:denominator-thm}. Section~\ref{ss:denominator-thm-app} consists of certain consequences of the denominator theorem. Among others, Theorem~\ref{t:thm2} is proved. In Section~\ref{s:c-vector}, we obtain categorical interpretations for $g$-vectors and $c$-vectors (Theorem~\ref{t:g-vector and index} and~\ref{t:thm4}).

\subsection*{Convention} 
Let $K$ be an algebraically closed field. Denote by $D=\Hom_K(-,K)$ the duality over $K$. Fix a positive integer $n$, we denote by $e_1,\cdots, e_n$ the standard basis of $\Z^n$. For a matrix $B$, denote by $B^{\opname{tr}}$ the transpose of $B$.
For an object $M$ in a category $\cc$, denote by $|M|$ the number of non-isomorphic indecomposable direct summands of $M$. Denote by $\add M$ the subcategory of $\cc$ consisting of objects which are finite direct sum of direct summands of $M$.  For an object $M$ in a triangulated category $\der$, denote by $\opname{thick}(M)$ the smallest triangulated subcategory of $\der$ containing $M$ which is closed under direct summands.

\section{Preliminaries}~\label{l:preliminaries}
\subsection{Recollection on cluster algebras}~\label{ss: cluster algebra}
We follow ~\cite{FominZelevinsky07}.
For an integer $x$, we use the notation $[x]_+=\max\{x,0\}$ and $\opname{sgn}(x)=\begin{cases}\frac{x}{|x|}& \text{if $x\neq 0$}\\ 0& \text{if $x=0$}\end{cases}$. A {\it semifield} $(\mathbb{P}, \oplus, \circ)$ is an abelian multiplicative group endowed with a binary operation of addition $\oplus$ which is commutative, associative and distributive with respect to the multiplication $\circ$ in $\mathbb{P}$. Let $J$ be a finite set of labels, the {\it tropical semifield} on variables $u_j (j\in J)$ is the abelian group freely generated by the variables $u_j$, endowed with the addition $\oplus$ defined by
\[\prod_{j}u_j^{a_j}\oplus \prod_{j}u_j^{b_j}=\prod_{j}u_j^{\operatorname{min}(a_j,b_j)}.
\]

Let $m, n$ be positive integers with $m\geq n$ and $\mathbb{P}$ the tropical semifield on variables $x_{n+1}, \cdots, x_m$. Let $\Q\mathbb{P}$ be the group algebra of $\mathbb{P}$ with rational coefficients and $\mathcal{F}$ the field of rational functions in $n$ variables with coefficients in $\Q\mathbb{P}$. Recall that an $n\times n$ integer matrix $B$ is {\it skew-symmetrizable} if there is a diagonal matrix $D=\opname{diag}\{d_1, \cdots, d_n\}$ with positive integers $d_1, \cdots, d_n$ such that $DB$ is skew-symmetric. In this case, we call $D$ the {\it skew-symmetrizer} of $B$. A {\it seed} in $\mathcal{F}$ is a pair $(\tilde{B}, \mathrm{x})$, where
\begin{itemize}
\item $\tilde{B}$ is an $m\times n$ integer matrix, whose {\it principal part} $B$, the top $n\times n$ submatrix, is skew-symmetrizable, and
\item $\mathrm{x}=\{x_1, \cdots, x_n\}$ is a free generating set of the field $\mathcal{F}$.
\end{itemize}
We refer to $\mathrm{x}$ as the {\it cluster} and to $\tilde{B}$ as the {\it exchange matrix} of the seed $(\tilde{B}, \mathrm{x})$.

For  any $k=1,\cdots, n$, the {\it seed mutation} $\mu_k(\tilde{B}, \mathrm{x})$ of $(\tilde{B}, \mathrm{x})$ in direction $k$ is the seed $(\tilde{B}', \mathrm{x}')$, where
\begin{itemize}
\item[$\bullet$] an integer $m\times n$ matrix $\tilde{B}'=\mu_k(\tilde{B})=(b_{ij}')$ is given by setting
\[b_{ij}'=\begin{cases}-b_{ij}& \text{if $i=k$ or $j=k$;}\\ b_{ij}+\opname{sgn}(b_{ik})[b_{ik}b_{kj}]_+&\text{else.}\end{cases}
\]
\item[$\bullet$] the cluster $\mathrm{x}'$ is obtained from $\mathrm{x}$ by the {\it exchange relation}: replacing the element $x_k$ with
    \[x_k'=\frac{1}{x_k}(\prod_{i=1}^mx_i^{[b_{ik}]_+}+\prod_{i=1}^mx_i^{[-b_{ik}]_+}).
    \]
\end{itemize}

Let $\mathbb{T}_n$ be the $n$-regular tree, whose edges are labeled by the numbers $1,2,\cdots, n$,
so that the $n$ edges emanating from each vertex carry different labels.  A {\it cluster pattern} is the assignment of a seed $(\tilde{B}_t, \mathrm{x}_t)$ to each vertex $t$ of $\mathbb{T}_n$ such that the seeds
assigned to vertices $t$ and $t'$ linked by an edge labeled $k$ are obtained from each other
by the seed mutation $\mu_k$. For a given initial seed $(\tilde{B}, \mathrm{x})$, a cluster pattern is uniquely determined by an assignment
of $(\tilde{B}, \mathrm{x})$ to a root vertex $t_0\in \mathbb{T}_n$.

Fix a cluster pattern determined by assigning the initial seed $(\tilde{B}, \mathrm{x})$ to the root vertex $t_0$. Let $(\tilde{B}_t, \mathrm{x}_t)$ be the seed associated to $t\in \mathbb{T}_n$.
The {\it clusters} associated with $(\tilde{B}, \mathrm{x})$ are the sets $\mathrm{x}_t$ for each $t\in \mathbb{T}_n$. The {\it cluster variables} are the elements of the clusters.  A {\it cluster monomial} is a monomial in cluster variables all of which belong to the same cluster. 
The {\it cluster algebra $\ca(\tilde{B})=\ca(\tilde{B},\mathrm{x})$ with coefficients} is
the $\Z\mathbb{P}$-subalgebra of $\mathcal{F}$ generated by the cluster variables. A cluster algebra is {of \it skew-symmetric type} if the principal part of its initial exchange matrix is skew-symmetric.

  Now to any vertex $t_0\in\mathbb{T}_n$ and any skew-symmetrizable $n\times n$ matrix $B$, we associate the skew-symmetrizable matrix pattern of format $2n\times n$ such that the initial exchange matrix $\tilde{B}$ has principal part $B$ and its {\it coefficient part}, the bottom $n\times n$ submatrix, is the identity matrix. We refer to this pattern as  the {\it principal coefficients pattern} and to the associated cluster algebra $\ca(\tilde{B})$ as {\it a cluster algebra with principal coefficients}. For each $t\in \mathbb{T}_n$, let $C_t$ be the coefficient part of the exchange matrix $\tilde{B}_t$. Each column vector of $C_t$ is called a {\it $c$-vector} of the cluster algebra $\ca(\tilde{B})$ and $C_t$ is the {\it $C$-matrix} of $\ca(\tilde{B})$ at vertex $t$.
 It has been conjectured by Fomin-Zelevinsky~\cite{FominZelevinsky07} that each $c$-vector of $\ca(\tilde{B})$ is {\it sign-coherent}, that is either all entries non-negative or all entries non-positive. The sign-coherence conjecture has been established in~\cite{GHKK} recently in a full generality. We refer to ~\cite{DWZ10,Fu17,Nagao13, Plamondon11} for a proof of the sign-coherence conjecture for skew-symmetric cases and ~\cite{Demonet10} for a proof for certain skew-symmetrizable cluster algebras.
 
 Let $x_1, \cdots, x_n$ be the initial cluster variables of the cluster algebra $\ca(\tilde{B})$ with principal coefficients. By the Laurent phenomenon, each cluster variable $x_{j,t}$ of the cluster $\mathrm{x}_t$ can be expressed as
\[X_{j,t}(x_1, \cdots, x_{2n})\in \Z[x_1^{\pm}, \cdots, x_n^{\pm}, x_{n+1}, \cdots, x_{2n}].
\]
Set $\deg x_i=e_i$ for $1\leq i\leq n$ and $\deg x_{n+j}=-b_j$ for $1\leq j\leq n$, where $b_j$ is the $j$-th column of the principal part of the initial exchange matrix $\tilde{B}$. In particular, $\Z[x_1^{\pm}, \cdots, x_n^{\pm}, x_{n+1}, \cdots, x_{2n}]$ is a $\Z^n$-graded ring. Fomin and Zelevinsky~\cite{FominZelevinsky07} proved that each $X_{j,t}$ is homogeneous with respect to the $\Z^n$-grading and defined 
\[\deg X_{j,t}=g_{j,t}^{\tilde{B}, t_0}=(g_{1j},\cdots, g_{nj})^{\opname{tr}}\in \Z^n\]
to be the {\it $g$-vector} of the cluster variable $x_{j,t}$. The matrix $G_t=(g_{ij})_{i,j=1}^n$ is called the {\it $G$-matrix} of $\ca(\tilde{B})$ at vertex $t\in \mathbb{T}_n$. 

 Let $B=(b_{ij})$ be an $n\times n$ skew-symmetrizable  matrix. The {\it Cartan counterpart} $A(B)=(c_{ij})$ of $B$ is an $n\times n$ integer matrix such that $c_{ii}=2$ for $i=1, \cdots, n$ and $c_{ij}=-|b_{ij}|$ for $i\neq j$. In particular, $A(B)$ is a generalized Cartan matrix. We refer to \cite{Kac} for the classification of generalized Cartan matrices.
\begin{definition}
The cluster algebra $\ca(\tilde{B})$ is of {\it type $\mathrm{C}$} if there is a vertex $t\in \mathbb{T}_n$ such that the Cartan counterpart of the principal part of $\tilde{B}_t$ is a generalized Cartan matrix of type $\mathrm{C}$.
\end{definition}
 Cluster algebras of type $\mathrm{C}$ have been investigated via different viewpoints. 
 It is well-known that a cluster determines its seed for a cluster algebra of type $\mathrm{C}$~ ({\it cf.}~\cite{FominZelevinsky03}).
 The following proposition summarizes certain results concerning $c$-vectors and $g$-vectors for cluster algebras of type $\mathrm{C}$ and  we refer to~\cite{NS14} for more results on cluster algebras of finite type.
\begin{proposition}~\label{p:property-c-type}
Let $\ca(\tilde{B})$ be a cluster algebra of type $\mathrm{C}$ with principal coefficients, where $\tilde{B}\in M_{2n\times n}(\Z)$. Assume that $\tilde{B}$ has principal part $B$. Then
\begin{itemize}
\item[(1)] Each $c$-vector of $\ca(\tilde{B})$ is sign-coherent;
\item[(2)] Let $\operatorname{cv}(\ca(\tilde{B}))$ be the set of $c$-vectors of $\ca(\tilde{B})$, then $|\operatorname{cv}(\ca(\tilde{B}))|=2n^2$;
\item[(3)] For each $t\in \mathbb{T}_n$, $G_t^{\opname{tr}}DC_t=D$, where $D$ is the skew-symmetrizer of $B$;
\item[(4)] Let $t_0\frac{~k~}{}t_1$ be two adjacent vertices in $\mathbb{T}_n$, and let $\tilde{B}_1=\left(\begin{array}{c}\mu_k(B)\\ E_n\end{array}\right)$. Then for any $t\in \mathbb{T}_n$, the $g$-vectors $g_{l,t}^{\tilde{B}, t_0}=(g_1,\cdots,g_n)^{\opname{tr}}$ and $g_{l,t}^{\tilde{B}_1, t_1}=(g_1',\cdots, g_n')^{\opname{tr}}$ are related as follows:
    \[g_j'=\begin{cases}-g_k &\text{if}~ j=k;\\ g_j+[b_{jk}]_+g_k-b_{jk}\min(g_k,0) & \text{else}.\end{cases}
    \]
\end{itemize}
\end{proposition}
\begin{proof}
The statements $(1)$ and $(4)$ follow from ~\cite{Demonet10,GHKK}. Part $(3)$ follows from part $(1)$ and ~\cite{NZ12}. Part $(2)$ is a result of~\cite{NS14}.
\end{proof}

\subsection{Recollection on  $\tau$-tilting theory} We follow~\cite{Adachi-Iyama-Reiten}.
Let $A$ be a basic finite dimensional algebra over $K$ and $\mod A$  the category of finitely generated right $A$-modules.  Let $\proj A$ be the full subcategory of $\mod A$ consisting of finitely generated projective $A$-modules. Denote by $\tau$ the Auslander-Reiten translation of $\mod A$.
Let $S_1, \cdots, S_n$ be all the pairwise non-isomorphic simple $A$-modules and $P_1, \cdots, P_n$ the corresponding projective covers respectively.

 An $A$-module $M$ is {\it $\tau$-rigid} provided $\Hom_A(M, \tau M)=0$. Let \[P_1^M\xrightarrow{f}P_0^M\to M\to 0\] be a minimal projective presentation of $M$, then $M$ is $\tau$-rigid if and only if $\Hom_A(f, M)$ is surjective.
 A  {\it $\tau$-rigid pair} is a pair of $A$-modules $(M, P)$ with $M\in \mod  A$ and $P\in \proj A$, such that $M$ is $\tau$-rigid and $\Hom_A(P, M)=0$. A basic $\tau$-rigid pair $(M,P)$ is a {\it support $\tau$-tilting pair} if $|M|+|P|=|A|=n$. In this case, $M$ is a {\it support $\tau$-tilting } $A$-module and $P$ is uniquely determined by $M$.  In particular, if $|M|=|A|$, then $M$ is a {\it $\tau$-tilting} $A$-module. 
 We also call a $\tau$-rigid pair {\it indecomposable} if $|M|+|P|=1$. It was shown in~\cite{Adachi-Iyama-Reiten} that each basic $\tau$-rigid pair can be completed to a support $\tau$-tilting pair. On the other hand, for a basic support $\tau$-tilting pair $(M,P)$, $M$ is a $\tau$-tilting  $A/\langle e_P\rangle$-module, where $e_P$ is the idempotent of $A$ associated to $P$.

 The support $\tau$-tilting $A$-modules have close relation with  $2$-term silting objects in the perfect derived category $\per A$ of $A$. Denote by $\Sigma$ the suspension functor of $\per A$. Recall that an object $Q\in \per A$ is {\it presilting} if $\Hom_{\per A}(Q, \Sigma^iQ)=0$ for all $i>0$. A presilting  object $Q\in \per A$ is {\it silting} if moreover $\opname{thick}(Q)=\per A$. Each basic silting object has exactly $|A|$ indecomposable direct summands and gives rise to a  $\Z$-basis of the Grothendieck group $\opname{G}_0(\per A)$. A silting object $Q$ is {\it $2$-term silting } with respect to $A$ if there is a triangle in $\per A$
 \[P_1^Q\to P_0^Q\to Q\to \Sigma P_1^Q~\text{with} ~P_1^Q, P_0^Q\in \proj A.
 \]
 Let $\opname{s\tau-tilt} A$ be the set of isomorphism classes of support $\tau$-tilting $A$-modules and $\opname{2-silt} A$ the set of isomorphism classes of $2$-term silting objects of $\per A$. In~\cite{Adachi-Iyama-Reiten}, Adachi {\it et al.} established the following bijection.
 \begin{theorem}
 There is a bijection between $\opname{s\tau-tilt} A$ and $\opname{2-silt} A$ given by
 \[M\in \opname{s\tau-tilt} A\mapsto (P_1^M\oplus P\xrightarrow{(f,0)}P_0^M)\in \opname{2-silt} A,
 \]
 where $P_1^M\xrightarrow{f}P_0^M\to M\to 0$ is a minimal projective presentation of $M$ and $(M,P)$ is the support $\tau$-tilting pair determined by $M$.
 \end{theorem}

 Let $M$ be a $\tau$-rigid $A$-module and $P_1^M\to P_0^M\to M\to 0$ a minimal projective presentation of $M$. The {\it index} of $M$ is defined to be
 \[\opname{ind}(M)=[P_0^M]-[P_1^M]\in \go(\per A),
 \]
 where $[X]$ stands for the image of $X$ in the Grothendieck group $\go(\per A)$.
 The {\it $g$-vector $g(M)$} of $M$ is the coordinate vector of $\opname{ind}(M)$ with respect to the canonical basis $[P_1], \cdots, [P_n]$ of $\go(\per A)$. It is known that different $\tau$-rigid modules have different $g$-vectors~({\it cf.}~\cite{Adachi-Iyama-Reiten, DehyKeller08}). We may extend the definition of $g$-vectors to any $\tau$-rigid pair $(M,P)$ by setting $g(M,P)=g(M)-g(P)\in \Z^n$.

For a given basic support $\tau$-tilting pair $(M,P)$ with decomposition of indecomposable $\tau$-rigid pairs, say $(M,P)=X_1\oplus\cdots\oplus X_n$, the {\it $\opname{G}$-matrix} of $M$ with respect to the decomposition is defined as 
\[G_M(A)=(g(X_1), \cdots, g(X_n)),
\]
which is invertible over $\Z$ ({\it cf.}~\cite{Adachi-Iyama-Reiten}). The {\it $\opname{C}$-matrix $C_M(A)$} associated to $M$ with respect to the decomposition is the inverse of the transpose of the $\opname{G}$-matrix $G_M(A)$, and its column vectors are called  {\it $c$-vectors} of $A$ associated to $M$~({\it cf.}~\cite{Fu17}). This notion is closely related to the $c$-vectors of cluster algebras provided that $A$ is a cluster-tilted algebra. We also remark that the $c$-vectors associated to $M$ do not depend on the decomposition of the support $\tau$-tilting pair $(M,P)$.

 Let $\der^b(\mod A)$ be the bounded derived category of $\mod A$ and $\go(\der^b(\mod A))$ the Grothendieck group of $\der^b(\mod A)$.
  Note that, as $A$ is finite dimensional, we have $\per A\subseteq\der^b(\mod A)$.
 The Euler bilinear form $\langle-, -\rangle:\opname{G}_0(\per A)\times \opname{G}_0(\der^b(\mod A))\to K$ given by
 \[\langle[P], [X]\rangle=\sum_{i\in \Z}(-1)^i\dim_K \Hom_{\der^b(\mod A)}(P, \Sigma^i X), 
 \]
 is non-degenerate, where~ $P\in \per A$~and~$X\in \der^b(\mod A)$. For each basic support $\tau$-tilting module $M$, let $Q_M=\bigoplus\limits_{i=1}^nQ_i^M$ be the corresponding $2$-term silting object in $\per A$ which gives rise to a basis $[Q_1^M], \cdots, [Q_n^M]$ of $\opname{G}_0(\per A)$. Let $[Q_1^M]^*, \cdots, [Q_n^M]^*\in \opname{G}_0(\der^b(\mod A))$ be the basis dual to $[Q_1^M], \cdots, [Q_n^M]$ with respect to  the Euler bilinear form $\langle-,-\rangle$.
Then the $c$-vectors associated to $M$ coincide with the coordinate vectors of the dual basis $[Q_1^M]^*, \cdots, [Q_n^M]^*$ with respect to  the canonical basis $[S_1], \cdots, [S_n]$ of $\opname{G}_0(\der^b(\mod A))$ given by simple $A$-modules. 

Recall that a non-zero integer vector is {\it positive} if all of its entries are nonnegative.
Each $c$-vector of $A$ is sign-coherent and each positive $c$-vector is the dimension vector of an indecomposable $A$-module.
We have the following criterion of positive $c$-vectors~({\it cf.}~Proposition 3.3 of~\cite{Fu17}).
\begin{proposition}~\label{p:criterion}
A vector $\underline{c}\in \Z^{n}$ is a positive $c$-vector for a finite dimensional  $K$-algebra $A$ if and only if there is a $2$-term silting object $Q\in \per A$ and an indecomposable $A$-module $M$ such that \[\Hom_{\der^b(\mod A)}(Q, \Sigma^i M)=\begin{cases}K&i=0;\\ 0&\text{otherwise.}\end{cases}\]
 In this case, $\underline{c}=\dimv M$.
\end{proposition}

For a $\tau$-tilting $A$-module $M=\bigoplus\limits_{i=1}^nM_i$ with decomposition of indecomposable modules, we define 
\[D_M(A):=(\dimv M_1,\cdots, \dimv M_n)\in M_n(\mathbb{Z})
\]
and call $D_M(A)$ the {\it $\opname{D}$-matrix} of $M$ with respect to the decomposition $M=\bigoplus\limits_{i=1}^nM_i$. Recall that we also have a $\opname{G}$-matrix $G_M(A)$ associated to $M$ with respect to the decomposition $M=\bigoplus\limits_{i=1}^nM_i$. The following result gives a relation between $G_M(A)$ and $D_M(A)$ for a $\tau$-tilting $A$-module $M$~({\it cf.} Proposition~5.3 of ~\cite{Adachi-Iyama-Reiten}).
\begin{lemma}~\label{l:g-d-duality}
Let $A$ be a finite dimensional $K$-algebra and $M=\bigoplus\limits_{i=1}^nM_i$ a $\tau$-tilting $A$-module. Then
\[G_M(A)^{\opname{tr}}D_M(A)=C(\End_A(M)),
\]
where $C(\End_A(M))$ is the Cartan matrix of the endomorphism algebra $\End_A(M)$.
\end{lemma}
\begin{proof}
Let $P_1^{M_i}\to P_0^{M_i}\to M_i\to 0$ be a minimal projective presentation of $M_i$. Then $M=\bigoplus\limits_{i=1}^nM_i$ is a $\tau$-tilting module implies that for any $1\leq i,j\leq n$, we have a short exact sequence
\[0\to \Hom_A(M_i,M_j)\to \Hom_A(P_0^{M_i}, M_j)\to \Hom_A(P_1^{M_i}, M_j)\to 0.\]
Now the result follows from the fact that the $(i,j)$-entry of $C(\End_A(M))$  is $\dim_K\Hom_A(M_i, M_j)$.
\end{proof}

\subsection{Cluster tubes}~\label{ss:cluster-tubes}
We follow~\cite{BuanMarshVatne}. Fix a non-negative integer $n$. Let $\ct$ be a tube of rank $n+1$ and $\cc$ the associated cluster tube of rank $n+1$. The Auslander-Reiten translation $\tau$ of $\der^b(\ct)$ induces the Auslander-Reiten translation of $\cc$, which we will still denote it by $\tau$. Note that, as $\cc$ is a $2$-Calabi-Yau triangulated category, we have $\tau=\Sigma$.
Let $X$ and $Y$ be indecomposable objects of $\ct$. Recall that we have identified the objects of $\ct$ with the ones of $\cc$.  By definition of $\cc$ and the fact that $\ct$ is hereditary, we have
\[\Hom_{\cc}(X, Y)=\Hom_{\der^b(\ct)}(X, Y)\oplus \Hom_{\der^b(\ct)}(X, \tau^{-1}\circ \Sigma Y).
\]
Following~\cite{BuanMarshVatne}, morphisms in $\Hom_{\der^b(\ct)}(X, Y)$ are called {\it $\ct$-maps} from $X$ to $Y$ and morphisms in $\Hom_{\der^b(\ct)}(X, \tau^{-1}\circ \Sigma Y)$ are called {\it $\der$-maps} from $X$ to $Y$. Each morphism from $X$ to $Y$ in $\cc$ can be written as the sum of a $\ct$-map with a $\der$-map. It is also well-known that the composition of two $\ct$-maps is also a $\ct$-map, the composition of a $\ct$-map with a $\der$-map is a $\der$-map, and the composition of two $\der$-maps is zero, and no $\ct$-map can factor through a $\der$-map ({\it cf.} \cite{BMR}).

The following lemma is useful ({\it cf.} Lemma $2.1$ of ~\cite{BuanMarshVatne}).
\begin{lemma}~\label{l:compute-morphism-cluster-tube}
Let $X, Y$ be indecomposable objects of $\ct$,  we have
\[\Hom_{\cc}(X, Y)\cong \Hom_{\ct}(X, Y)\oplus D\Hom_{\ct}(Y, \tau^2 X).
\]
\end{lemma}
An obvious consequence is that the existence of a non-zero $\der$-map from $X$ to $Y$ is equivalent to the existence of a non-zero $\ct$-map from $Y$ to $\tau^2X$.

 For each indecomposable object $X=(a, b)\in \ct$, the infinite sequence  of irreducible maps
\[\mathbf{R}_{(a,b)}=(a, b)\to (a, b+1)\to \cdots \to (a, b+j)\to \cdots
\]
is called a {\it ray} starting in $X$ and the infinite sequence of irreducible maps 
\[\mathbf{C}_{(a,b)}=\cdots\to (a-j, b+j)\to \cdots \to (a-1, b+1)\to (a,b)
\]
 is called a {\it coray} ending in $X$. We also denote by \[\mathbf{R}_{(a,b)}=\{(a, b+j)~|~j\geq 0\}~\text{and} ~\mathbf{C}_{(a,b)}=\{(a-j, b+j)~|~j\geq 0\}.\]

For each indecomposable object $X\in \ct$ with  length $l(X)\leq n$, the {\it wing $\cw_{X}$ determined by $X$} is the set of indecomposables whose position in the AR-quiver is in the triangle with $X$ on top.
We also denote them by $X^{\sqsubset}$ the support of $\Hom_{\ct}(X,-)$ in $\ct$. Namely, $X^{\sqsubset}$ consists of indecomposable objects $Y$ of $\ct$ such that $\Hom_{\ct}(X,Y)\neq 0$. Dually, we may define $~^{\sqsupset}X$ to be the support of $\Hom_{\ct}(-,X)$ in $\ct$. As in \cite{BuanMarshVatne}, the $\Hom$-hammock of an indecomposable object $X$ is the support of $\Hom_\cc(X, -)$. By Lemma~\ref{l:compute-morphism-cluster-tube}, we clearly know  that an indecomposable object $Y\in \cc$ satisfies that $\Hom_{\cc}(X,Y)\neq 0$ if and only if $Y\in X^{\sqsubset}\cup~^{\sqsupset}\tau^2X$ ({\it cf.} Figure 1.).
\vspace{0.2cm}
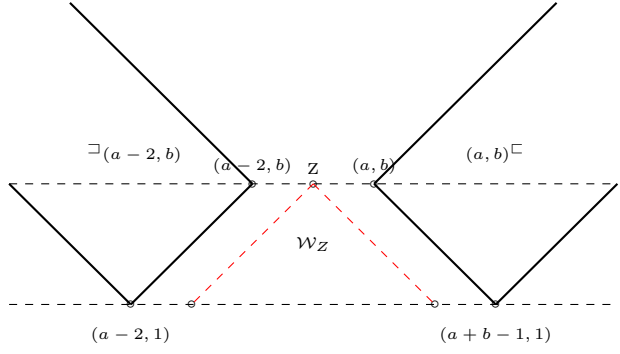
\begin{figure}
\begin{tikzpicture}[
    scale=4,axis/.style={ very thick, ->, >=stealth'},
    arrow/.style={thick, ->, >=stealth'},
    important line/.style={thick},
    dashed line/.style={dashed, thin},
    pile/.style={thick, ->, >=stealth', shorten <=2pt, shorten
    >=2pt},
    every node/.style={color=black}
    ]
  
    \draw[dashed line](-1,0) --(1,0)[right];
    \draw[dashed line](-1,.4)--(1,.4)[right];
    \draw[important line](-.6,0)--(-.2,.4)[right]node[above]{\tiny{$(a-2,b)$}};
    \node at(-.6,0){\tiny{$\circ$}};
    \node at(-.2,.4){\tiny{$\circ$}};
    \node at (-.6,-.1){\tiny{$(a-2,1)$}};
    \node at(0,.4){\tiny{$\circ$}};
    \node at(0,.45){\tiny{Z}};
    \node at(.4,0){\tiny{$\circ$}};
    \node at(-.4,0){\tiny{$\circ$}};
    \node at(0,.2){\tiny{$\cw_Z$}};
    \draw[dashed line,red](0,.4)--(-.4,0)[left];
    \draw[dashed line,red](0,.4)--(.4,0)[left];
    \draw[important line](.6,0)--(.2,.4)[left]node[above]{\tiny{$(a,b)$}};
    \node at(.2,.4){\tiny{$\circ$}};
    \node at(.6,0){\tiny{$\circ$}};
    \node at(.6, -.1){\tiny{$(a+b-1,1)$}};
    \draw[important line](-.6,0)--(-1,.4)[right];
    \draw[important line](-.2,.4)--(-.8,1)[right];
    \draw[important line](.6,0)--(1,.4)[right];
    \draw[important line](0.2, 0.4)--(.8,1)[right];
      \node at(-.6,.5){\tiny{$~^\sqsupset(a-2,b)$}};
      \node at(.6,.5){\tiny{$(a,b)^{\sqsubset}$}};
\end{tikzpicture}
\caption{For an indecomposable $X=(a,b)$, the $\Hom$-hammock is illustrated by the full lines.}
\end{figure}

Recall that a rigid object $T\in \cc$ is called a {\it cluster-tilting object} if $\Hom_{\cc}(T, \Sigma Y)=0$ implies that $Y\in \add T$. It was proved in ~\cite{BuanMarshVatne} that the cluster tube $\cc$ has no cluster-tilting objects but only maximal rigid objects. Moreover, the following description of maximal rigid objects was given.
\begin{lemma}~\label{l:basic-property-cluster-tube}
\begin{itemize}
\item[(1)] An indecomposable object $M$ in $\cc$ is rigid if and only if it has length $l(M)\leq n$;
\item[(2)] Every basic maximal rigid object in $\cc$ has exactly $n$ indecomposable direct summands;
\item[(3)] Each basic maximal rigid object has exactly one indecomposable direct summand with length $n$;
\item[(4)] There is a bijection between the set of maximal rigid objects
in $\cc$ and the set 
\[\{tilting \ modules\ of\ K\vec{A}_{n}\}\times\{1,2,\cdots, n+1\},\]
where $\vec{A}_{n}$  is a linearly oriented quiver with the Dynkin diagram $A_{n}$ as its underlying
graph.
\end{itemize}
\end{lemma}
 \begin{remark}
If $n=0$, the zero object is the unique maximal rigid object in $\cc_1$. If $n=1$, there are precisely two basic maximal rigid objects in $\cc_2$, {\it i.e.} $(1,1)$ and $(1,2)$.
\end{remark}
The following result is a direct consequence of Lemma~\ref{l:compute-morphism-cluster-tube} and the $\Hom$-hammock.
\begin{lemma}~\label{l:dim}
\begin{itemize}
\item[(1)] Let  $M, N$ be two indecomposable rigid objects with $l(N)=n$. We have 
\[\dim_K\Hom_\cc(N, M)=\begin{cases}2&\text{if $M\not\in \cw_{\tau N}$;}\\ 0& \text{if $M\in \cw_{\tau N}$}.\end{cases}
\]
\item[(2)] Let $X$ be an indecomposable rigid object. For any indecomposable rigid object $Y\in \cw_{\tau^{-1}X}$, we have $\Hom_\cc(Y,X)=0$.
\end{itemize}
\end{lemma}

The following lemma collects certain results on morphisms related to indecomposable rigid objects.
\begin{lemma}~\label{l:factor-morphism}
Let $X$ be an indecomposable rigid object in $\cc$. 
\begin{itemize}
\item[(i)] For each indecomposable object $Y\in X^{\sqsubset}$, we have $\Hom_{\ct}(X,Y)\cong K$.
\item[(ii)] For each indecomposable object $Z\in ~^{\sqsupset}X$, we have $\Hom_{\ct}(Z,X)\cong K$.
\item[(iii)] Let $Y, Z$ be indecomposable rigid objects such that $Y, Z\in X^{\sqsubset}$ and $Y\in Z^{\sqsubset}$. Let $f:Z\to Y$ be a non-zero $\ct$-map, then each $\ct$-map from $X$ to $Y$ factors through $f$.
\item[(iv)] Let $Y,Z$ be indecomposable rigid objects such that $X,Z\in \tau^{-2}Y^{\sqsubset}$ and $Z\in X^{\sqsubset}$.
Let $g: Z\to Y$ be a non-zero $\der$-map, then each $\der$-map from $X$ to $Y$ factors through $g$.
\end{itemize}
\end{lemma}
\begin{proof}
The statements (i), (ii) and (iii) can be read directly from the AR-quiver of $\ct$. For (iv), we first note that $Z\in \tau^{-2}Y^{\sqsubset}$ implies that \[D\Hom_{\der^b(\ct)}(Z, \tau^{-1}\circ \Sigma Y)\cong\Hom_\ct(\tau^{-2}Y, Z)\cong K\]
 by $(i)$. Hence there exist non-zero $\der$-maps from $Z$ to $Y$. By $Z\in X^{\sqsubset}$ and $X\in \tau^{-2}Y^{\sqsubset}$, we deduce that $X\in\tau^{-2}Y^{\sqsubset}\cap ~^{\sqsupset} Z $. Let $h:X\to Z$ be a non-zero $\ct$-map. Applying (iii), we conclude that $h$ induces an isomorphism of vector spaces
 \[\Hom_\ct(\tau^{-2}Y, h):\Hom_\ct(\tau^{-2}Y, X)\xrightarrow{\sim} \Hom_\ct(\tau^{-2}Y, Z).
 \]
 On the other hand, we have the following commutative diagram
 \[\xymatrix{\Hom_\ct(\tau^{-2}Y, X)\ar[d]^{\cong}\ar[rrr]^{\Hom_\ct(\tau^{-2}Y, h)} &&&\Hom_\ct(\tau^{-2}Y, Z)\ar[d]^{\cong}\\
 D\Hom_{\der^b(\ct)}(X,\tau^{-1}\circ\Sigma Y)\ar[rrr]^{D\Hom_{\der^b(\ct)}(h,\tau^{-1}\circ\Sigma Y)}&&&D\Hom_{\der^b(\ct)}(Z,\tau^{-1}\circ\Sigma Y).
 }
 \]
 In particular, $h$ also induces an isomorphism
 \[\Hom_{\der^b(\ct)}(h, \tau^{-1}\circ\Sigma Y): \Hom_{\der^b(\ct)}(Z, \tau^{-1}\circ\Sigma Y)\xrightarrow{\sim} \Hom_{\der^b(\ct)}(X, \tau^{-1}\circ\Sigma Y).
 \]
 Now the result follows from the fact that \[\dim_K\Hom_{\der^b(\ct)}(Z, \tau^{-1}\circ\Sigma Y)=1=\dim_K\Hom_{\der^b(\ct)}(X, \tau^{-1}\circ\Sigma Y).\]
\end{proof}

For a given basic maximal rigid object $T=\overline{T}\oplus T_k$ in $\cc$ with an indecomposable direct summand $T_k$,  the {\it mutation $\mu_k(T)$ of $T$ at $T_k$} is a basic maximal rigid object obtained by replacing $T_k$ by another indecomposable object $T_k^*$.
The objects $T_k^*$ and $ T_k$ are related by the following {\it exchange triangles}
\[T_k^*\xrightarrow{f}U_{T_k, T\setminus T_k}\xrightarrow{g}T_k\to \Sigma T_k^* ~ \text{and}~ T_k\xrightarrow{f'} U'_{T_k, T\setminus T_k}\xrightarrow{g'}\ T_k^*\to\Sigma T_k,
\]
where $f$ and $f'$ are minimal left $\add \overline{T}$-approximations and $g$ and $g'$ are minimal right $\add \overline{T}$-approximations. In this case, $\overline{T}$ is called an {\it almost complete maximal rigid object} and $(T_k, T_k^*)$ is an {\it exchange pair} of $\cc$.

By using Lemma~\ref{l:basic-property-cluster-tube} (4) and the classification of tilting modules for a quiver of type $A_n$ with linear orientation, Zhou and Zhu~\cite{ZhouZhu14} obtained the following description of the exchange triangles in $\cc$. 
\begin{lemma}\label{l:exchange triangle in cluster tube}
Given two basic maximal rigid objects $T_k\oplus \overline{T}$ and $T_k^*\oplus \overline{T}$ in $\cc$ such that both $T_k$ and $T_k^*$ are indecomposable. Then $\dim_K\Ext_{\cc}^{1}(T_k,T_k^*)=1$ or $2$. Moreover, 
\begin{itemize}
\item[(1)]~if $\dim_K\Ext_{\cc}^{1}(T_k,T_k^*)=2$, then  $T_k$ and $T_k^*$ are of length $n$. Denote by $T_k=(a,n)$ and $T_k^*=(a+h,n)$ respectively with $1\leq a\leq n+1, 1\leq h\leq n$, 
 then the exchange triangles are of the following forms:
\[(a,n)\to (a+h,n-h)\oplus (a+h,n-h)\to(a+h,n)\to\Sigma(a,n)\]
\[(a+h,n)\to (a,h-1)\oplus (a,h-1)\to (a,n)\to\Sigma(a+h,n);\]

\item[(2)]
if $\dim_K\Ext^1_{\cc}(T_k,T_k^*)=1$, then $l(T_k)<n$ and $l(T_k^*)<n$.  Denote by $T_k=(a,b)$, then $T_k^*=(a+h,b-h+i)$, where $1\leq a\leq n+1, 1\leq b<n, 1\leq h\leq b, 1\leq i\leq n-b$, and 
the exchange triangles are of  the following forms:
\[(a,b)\to (a,b+i)\oplus(a+h,b-h)\to (a+h,b-h+i)\to \Sigma(a,b)\]
\[(a+h,b-h+i)\to (a+b+1,i-1)\oplus(a,h-1)\to (a,b)\to \Sigma(a+h,b-h+i).
\]
\end{itemize}
\end{lemma}

For each basic maximal rigid object $T=\bigoplus\limits_{i=1}^nT_i$, we define a matrix $B_T=(b_{ij})\in M_{n}(\Z)$ as follows
\[b_{ij}=\alpha_{U_{T_j, T\setminus T_j}}T_i-\alpha_{U'_{T_j, T\setminus T_j}}T_i,
\]
where $\alpha_YX$ denotes the multiplicity of $X$ as a direct summand of $Y$. Recall that we have the Gabriel quiver $Q_T$ of the endomorphism algebra $\End_{\cc}(T)$. 
The following result suggests another construction of the skew-symmetrizable matrix $B_T$, which is an easy consequence of the definitions of $B_T$ and $Q_T$.
\begin{lemma}~\label{l:skew-symmetrizable-matrix-via-quiver}
Let $T=\bigoplus\limits_{i=1}^nT_i$ be a basic maximal rigid object of $\cc$ with $l(T_1)=n$ and $Q_T$ its associated quiver. Let $B_T=(b_{ij})\in M_n(\Z)$ be the skew-symmetrizable matrix associated to $T$. Then for $i\neq j$, we have
\[b_{ij}=\begin{cases}|\{\text{arrows $T_i\to T_j$}\}|-|\{\text{arrows $T_j\to T_i$}\}|& j\neq 1;\\
2|\{\text{arrows $T_i\to T_1$}\}|-2|\{\text{arrows $T_1\to T_i$}\}|& j=1.
\end{cases}
\]
\end{lemma}

 It was proved in ~\cite{BuanMarshVatne} that $B_T$ is a skew-symmetrizable matrix and when an indecomposable summand of a maximal rigid object is exchanged, the change in the matrix is given by Fomin-Zelevinsky's mutation of matrices, {\it {i.e.}} $\mu_k(B_T)=B_{\mu_k(T)}$ for any basic maximal rigid object $T$. Moreover,  the following result was proved~({\it cf.}~ Theorem 3.5 of ~\cite{BuanMarshVatne}).
\begin{theorem}~\label{t:BMV-main-theorem}
Let $T$ be a basic maximal rigid object of $\cc$ and $\mathcal{A}_T:=\mathcal{A}(B_T)$ the associated cluster algebra of type $\mathrm{C}_n$.
There is a bijection $\mathbb{X}_?$ between the indecomposable rigid objects of $\cc$ and the cluster variables of $\mathcal{A}_T$. The bijection induces a bijection between the basic maximal rigid objects of $\cc$ and  the clusters of $\mathcal{A}_T$ such that $\Sigma T$ corresponds to the initial cluster of $\mathcal{A}_T$. Moreover, the bijection is compatible with mutations.
\end{theorem}

\subsection{Algebras arising from cluster tubes}~\label{s:endomorphism-algebra}
Let $T$ be a basic maximal rigid object of the rank $n+1$ cluster tube $\cc$. 
Recall that $\opname{pr}T$ is the  full subcategory of $\cc$ consisting of objects which are finitely presented by $T$.
A general result of ~\cite{ZhouZhu} implies that the rigid objects of $\cc$ belong to $\opname{pr}T$.

Denote by $\Gamma:=\End_{\cc}(T)$ the endomorphism algebra of $T$ and $\mod \Gamma$ the category of finitely generated right $\Gamma$-modules. Recall that the functor \[F:=\Hom_{\cc}(T,-):\cc\to \mod \Gamma\]
induces an equivalence of categories \[F: \opname{pr}T/\add \Sigma T\xrightarrow{\sim} \mod \Gamma,\] where $\opname{pr}T/\add \Sigma T$ is the additive quotient of $\opname{pr} T$ by morphisms factorizing through $\add \Sigma T$. Moreover, the restriction of the functor $F$ to the subcategory $\add T$ yields an equivalence between $\add T$ and the category of finitely generated projective $\Gamma$-modules. The following bijection between the set of basic maximal rigid objects of $\cc$ and the set of basic support $\tau$-tilting modules of $\Gamma$ has been established in~\cite{LiuXie,ChangZhangZhu}. Namely,
\begin{theorem}~\label{t:LiuXie-ChangZhangZhu}
The functor $F$ yields a bijection between the basic maximal rigid objects of $\cc$ and the basic support $\tau$-tilting $\Gamma$-modules. 
\end{theorem}

Under the identification of $\mod \Gamma$ with $\opname{pr} T/\add \Sigma T$, we deduce that there are exactly $n^2$ indecomposable $\tau$-rigid $\Gamma$-modules by Lemma ~\ref{l:basic-property-cluster-tube} (1). Namely, the indecomposable $\tau$-rigid $\Gamma$-modules are precisely $F(X)=\Hom_{\cc}(T,X)$ for indecomposable object $X\in \cc$ with $~ l(X)\leq n$ and $X\not\in \add \Sigma T$. 

\begin{lemma}~\label{l:morphism of quasi length n}
Let $T$ be a basic maximal rigid object in $\cc$ and $\Gamma=\End_{\cc}(T)$.  For any indecomposable object $X\not\in \add\Sigma T$ with $ ~l(X)=n$, we have \[\dim_K\Hom_\Gamma(F(X), F(X))=2.\]
\end{lemma}
\begin{proof}
By definition, we have \[\Hom_{\cc}(X,X)=\Hom_{\der^b(\ct)}(X,X)\oplus \Hom_{\der^b(\ct)}(X,\tau^{-1}\Sigma X).\] In particular, $\dim_K\Hom_{\cc}(X,X)=2$. Recall that we have an equivalence $F:\opname{pr} T/\Sigma T\to \mod \Gamma$. Since $X$ is indecomposable and $X\not\in \add \Sigma T$, the identity morphism $1_X\in \Hom_{\ct}(X,X)$ does not factor through $\Sigma T$. Hence it suffices to show that each nonzero morphism $f\in \Hom_{\der^b(\ct)}(X,\tau^{-1}\Sigma X)$ does not factor through $\Sigma T$.  Otherwise, there exist morphisms $g_1\in \Hom_{\der^b(\ct)}(X,\Sigma T)$, $ f_1\in \Hom_{\der^b(\ct)}(\Sigma T, \tau^{-1}\Sigma X)$, $ g_2\in \Hom_{\der^b(\ct)}(X,\tau T)$ and $ f_2\in \Hom_{\der^b(\ct)}(\tau T, \tau^{-1}\Sigma X)$ such that $f=f_1\circ g_1+f_2\circ g_2$. We claim that $f_i\circ g_i=0$ for $i=1,2$. Indeed, for $i=1$, we consider the following commutative diagram
\[\xymatrix{\Hom_{\der^b(\ct)}(X,\Sigma T)\ar[d]^{\wr}\ar[rrr]^{\Hom_{\der^b(\ct)}(X,f_1)}&&&\Hom_{\der^b(\ct)}(X,\tau^{-1}\Sigma X)\ar[d]^{\wr}\\ D\Hom_{\der^b(\ct)}(T,\tau X)\ar[rrr]^{D\Hom_{\der^b(\ct)}(\Sigma^{-1} f_1,\tau X)}&&&D\Hom_{\der^b(\ct)}(\tau^{-1}X, \tau X).}
\]
To show that $f_1\circ g_1=0$, it suffices to show $\Hom_{\der^b(\ct)}(\Sigma^{-1} f_1, \tau X)=0$ and the later one follows from the fact that the image of any nonzero morphism $h\in\Hom_{\ct}(\tau^{-1}X, \tau X)$ is a simple object of $\ct$. Similarly, one can show that $f_2\circ g_2=0$, which contradicts the assumption that $f$ is nonzero.
\end{proof}

Let  $Q=(Q_0,Q_1)$ be a finite quiver with vertex set $Q_0$ and arrow set $Q_1$. 
 For an arrow $\alpha:i\to j\in Q_1$, $s(\alpha)=i$ is the {\it source} of $\alpha$ and $t(\alpha)=j$ is the {\it target} of $\alpha$. An oriented cycle of $Q$ of length $m$ is a path $c=\alpha_m\alpha_{m-1}\cdots \alpha_1$
such that $t(\alpha_i)=s(\alpha_{i+1})$ for $1\leq i<m$ and $s(\alpha_1)=t(\alpha_m)$, where $\alpha_1,\cdots, \alpha_m\in Q_1$.

Let $\mathcal{Q}_{n}$ be the set of quivers with $n$ vertices satisfying the following conditions:
\begin{itemize}
\item[(a)] All non-trivial minimal cycles of length at least $2$ in the underlying graph are oriented and of length $3$;
\item[(b)] Any vertex has at most four neighbors;
\item[(b)] If a vertex has four neighbors, then two of its adjacent arrows belong to one $3$-cycle, and the other two belong to another $3$-cycle;
    \item[(d)]If a vertex has three neighbors, then two of its adjacent arrows belong to one $3$-cycle, and the third one does not belong to any $3$-cycle;
        \item[(e)] There is a unique loop $\rho$ at a vertex $t$ which has one neighbor, or has two neighbors and its traversed by a $3$-cycle.
\end{itemize}
 The following result shows  that the endomorphism algebra $\End_{\cc}(T)$ of a basic maximal rigid object $T$ in $\cc$ is determined by its underlying quiver $Q_T$, which is due to ~\cite{Vatne11}~({\it cf. }also ~\cite{Yang12}).
\begin{theorem}~\label{t:algbra-maximal-tube}
 An algebra is the endomorphism algebra of a basic maximal rigid object in the cluster tube $\cc$ if and only if it is isomorphic to $KQ/I$ for some $Q\in \mathcal{Q}_{n}$, where $I$ is the ideal generated by the square of the unique loop $\rho$ and all paths of length $2$ in a $3$-cycle.
\end{theorem}

\begin{remark}~\label{r:cluster-tilted-A}
An algebra is a cluster-tilted algebra of type $A_n$ if and only if it is isomorphic to $KQ/J$ for some quiver with $n$ vertices satisfying $(a)-(d)$ and $J$ is the ideal generated by all paths of length $2$ in a $3$-cycle ({\it cf.}~\cite{BuanVatne}).
\end{remark}
In particular, Theorem~\ref{t:algbra-maximal-tube} implies that the endomorphism algebra of a basic maximal rigid object in $\cc$ is a gentle algebra. We refer to~\cite{AS87} for the definition of gentle algebras.  

Let $Q$ be a finite quiver and $I$ an admissible ideal of $Q$ such that $KQ/I$ is a gentle algebra.
An oriented cycle $c=\alpha_m\cdots\alpha_2\alpha_1$ of $Q$ is of {\it full relations} if $\alpha_{i+1}\alpha_i\in I$ for $i=1,\cdots, m-1$ and $\alpha_1\alpha_m\in I$. The following has been proved by Holm~\cite{Holm05}.
\begin{lemma}~\label{l:Cartan-matrix-gentle}
The Cartan matrix of a gentle algebra $KQ/I$ is degenerate if and only if $KQ/I$ admits at least one oriented cycle of even length with full relations.
\end{lemma}

In particular, the Cartan matrix of the endomorphism algebra of a basic maximal rigid object in the cluster tube $\cc$ is non-degenerate by Theorem~\ref{t:algbra-maximal-tube} and Lemma~\ref{l:Cartan-matrix-gentle}.
Note that the global dimension of the endomorphism algebra of a maximal rigid object in $\cc$ is always infinite.

\section{Rank vectors and mutations}~\label{s:denominator vectors}

\subsection{Exchange compatibility}~\label{ss:exchange-compatible}
The aim of this subsection is to show that each indecomposable rigid object in the cluster tube $\cc$ is exchange compatible.  We begin with some definitions introduced in~\cite{BMR}.

Let $\overline{T}$ be an almost complete basic maximal rigid object in $\cc$ and $(X, X^*)$  the associated exchange pair. In particular,  $\overline{T}\oplus X$ and $\overline{T}\oplus X^{*}$ are basic maximal rigid objects in $\cc$ with the following exchange triangles
\[X^{*}\xrightarrow{f} B \xrightarrow{g} X \xrightarrow{} \Sigma X^{*}~\text{and}~
X \xrightarrow{f'} B' \xrightarrow{g'} X^{*}\xrightarrow{}\Sigma X,\] 
where $B, B'\in \add \overline{T}$.
An indecomposable rigid object $M\in \cc$ is {\it compatible} with the exchange pair $(X,X^*)$, if either $X\cong \Sigma M$, $X^*\cong \Sigma M,$ or, if neither of these holds,
\[\dim_{K}\Hom_{\cc}(M,X) + \dim_{K} \Hom_{\cc}(M,X^*) =
\max\{\dim_{K} \Hom_{\cc}(M,B), \dim_{K}\Hom_{\cc}(M,B')\}.\]
If $M$ is compatible with every exchange pair $(X, X^*)$ in $\cc$, then $M$ is called {\it exchange compatible}.

By the $2$-Calabi-Yau property of $\cc$, it is not hard to see that 
an indecomposable rigid object $M$ of $\cc$ is compatible with the exchange pair $(X,X^*)$ if and only if either $X\cong \Sigma M$, $X^*\cong \Sigma M,$ or, if neither of these holds,
\begin{eqnarray*}
&&\dim_{K}\Hom_{\cc}(X,\Sigma^2M) + \dim_{K} \Hom_{\cc}(X^*,\Sigma^2M) \\
&=&
\max\{\dim_{K} \Hom_{\cc}(B,\Sigma^2M), \dim_{K}\Hom_{\cc}(B',\Sigma^2M)\}. 
\end{eqnarray*}
Since $\Sigma$ is an autoequivalence of $\cc$, it is clear that   $(\Sigma^2X,\Sigma^2X^*)$ is  an exchange pair if and only if $(X,X^*)$ is an exchange pair. Consequently, we obtain the following~({\it cf.}~\cite{AD}).
\begin{lemma}\label{l:compatible}
Let $M$ be an indecomposable rigid object of $\cc$, then $M$ is compatible with an exchange pair $(\Sigma^2X,\Sigma^2X^*)$ if and only if  the following holds whenever $\Sigma X\not\cong  M\not\cong\Sigma X^*:$
\[\dim_{K}\Hom_{\cc}(X,M) + \dim_{K} \Hom_{\cc}(X^*,M) =
\max\{\dim_{K} \Hom_{\cc}(B,M), \dim_{K}\Hom_{\cc}(B',M)\}.
\]
\end{lemma}
The following lemma gives a sufficient condition on compatibility.
\begin{lemma}\label{l:judge for exchangeable}
Let $(X, X^*)$ be an exchange pair of $\cc$ with exchange triangles 
\[X^{*}\xrightarrow{f} B \xrightarrow{g} X \xrightarrow{}\Sigma X^* ~\text{and}~
 X \xrightarrow{f'} B' \xrightarrow{g'} X^{*}\xrightarrow{}\Sigma X.
 \]  
 Let $M$ be an indecomposable rigid object in $\cc$. If each morphism from $M\oplus \Sigma M$ to $X$ factors through $g$ or each morphism from $M\oplus \Sigma M$ to $X^*$ factors through $g'$, then $M$ is {compatible} with $(X,X^*)$.
\end{lemma}
\begin{proof}
Let us assume that each morphism from $M\oplus \Sigma M$ to $X$ factors through $g$.
 Applying the functor $\Hom_\cc(M,-)$ to the triangle $X^{*}\xrightarrow{f} B \xrightarrow{g} X \xrightarrow{}\Sigma X^*$ yields a long exact sequence
\begin{eqnarray*}
\cdots\Hom_\cc(M,\Sigma^{-1}B)\xrightarrow{\Hom_\cc(M,\Sigma^{-1}g)}\Hom_\cc(M,\Sigma^{-1}X)\xrightarrow{~~~~}\Hom_\cc(M,X^*)\\
\xrightarrow{\Hom_\cc(M,f)}\Hom_\cc(M,B)\xrightarrow{\Hom_\cc(M,g)}\Hom_\cc(M,X)\xrightarrow{}\Hom_\cc(M,\Sigma X^*)\cdots.
\end{eqnarray*}
By the assumption, we clearly know that both $\Hom_\cc(M,g)$ and $\Hom_\cc(M,\Sigma^{-1}g)$  are surjective.  In particular, $\dim_K\Hom_\cc(M,X)+\dim_K\Hom_\cc(M,X^*)=\dim_K\Hom_\cc(M,B)$.
Hence $M$ is compatible with the exchange pair $(X,X^*)$. Similarly one can prove the other case.
\end{proof}

\begin{proposition}~\label{p:exchange-n}
Suppose  that $X=(a,n), 1\leq a\leq n$ and $(X,X^*)$ is an exchange pair of $\cc$. Then  any indecomposable rigid object $M$ in $ \cc$ is compatible with  $(X,X^*).$
\end{proposition}
\begin{proof}
According to Lemma \ref{l:exchange triangle in cluster tube}, we may assume that $X^*=(a+h,n)$ for some $1\leq h\leq n$. Moreover,  the exchange triangles are of the following forms
\[(a,n)\to (a+h,n-h)\oplus (a+h,n-h)\xrightarrow{(f_1,f_2)} (a+h,n)\to\Sigma(a,n)\]
\[(a+h,n)\to (a,h-1)\oplus (a,h-1)\xrightarrow{(g_1,g_2)} (a,n)\to\Sigma(a+h,n).\]
Without loss of generality, we may assume that $f_1$ and $g_1$ ({\it resp.} $f_2$ and $g_2$) are non-zero $\ct$-maps ({\it resp.} $\der$-maps).

By definition, if $M\cong \Sigma^{-1}X=(a+1,n)$ or $M\cong \Sigma^{-1}X^*=(a+h+1,n)$, then $M$ is compatible with the exchange pair $(X, X^*)$. For the remaining case, by Lemma~\ref{l:judge for exchangeable}, it suffices to prove the following claim:

\noindent{\bf Claim:}
 either each morphism from $M\oplus \Sigma M$ to $X^*$ factors through $(f_1,f_2): (a+h,n-h)\oplus (a+h,n-h)\xrightarrow{} (a+h,n)$ or each morphism from 
$M\oplus \Sigma M$ to $X$ factors through $(g_1,g_2): (a,h-1)\oplus (a,h-1)\xrightarrow{} (a,n)$. 

For an indecomposable rigid object $M$ which is not isomorphic to $(a+1,n)$ nor $(a+h+1,n)$, $M$ belongs to one of the following 7 subsets: $\cw_{(a+h+1,n-1)}$, $~^\sqsupset(a+h-1,1)\cap (a+h+2,n-1)^{\sqsubset}$, $~^\sqsupset(a+h,1)\cap (a+h+2,n)^{\sqsubset}$,
$~^\sqsupset(a+h,n-h)\backslash ~^\sqsupset(a+h,1)$,
$~^\sqsupset(a+h,n-h+1)\cap (a+2,n-1)^\sqsubset$,
$~^\sqsupset(a+h,n-h+2)\cap (a+2,n)$, $(a+h,n)^\sqsupset\backslash (a+h,n-h+2)^\sqsupset$.
We separate the remaining part into $7$ cases ~({\it cf.} Figure 2.).

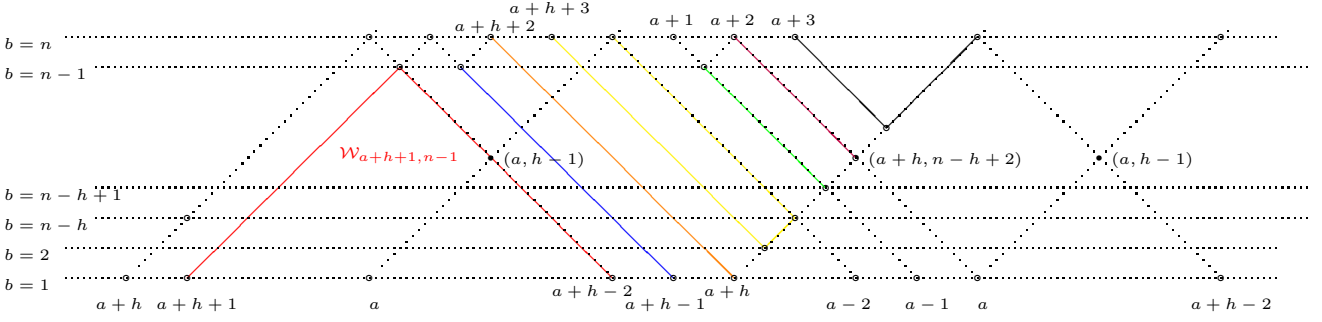
\begin{figure}
\setlength{\unitlength}{0.080cm}
\begin{picture}(220,50)

{\color{red}\put(20,0){\line(1,1){35}}
\put(55,35){\line(1,-1){35}}}
\put(20,0){\circle{1}}
\put(55,35){\circle{1}}
\put(45,20){\tiny{\color{red}$\cw_{a+h+1,n-1}$}}

{\color{blue}\put(65,35){\line(1,-1){35}}}
\put(65,35){\circle{1}}
\put(60,40){\circle{1}}

\put(70,40){\circle{1}}

{\color{orange}\put(70,40){\line(1,-1){40}}}
\put(110,0){\circle{1}}

{\color{yellow}
\put(80,40){\line(1,-1){35}}
\put(115,5){\line(1,1){5}}
\put(90,40){\line(1,-1){30}}
}

\put(100,40){\circle{1}}

\put(105,35){\circle{1}}
\put(125,15){\circle{1}}
{\color{green}\put(105,35){\line(1,-1){20}}}

\put(80,40){\circle{1}}

\put(110,40){\circle{1}}
\put(130,20){\circle{1}}
{\color{purple}\put(110,40){\line(1,-1){20}}}

\put(115,5){\circle{1}}

\put(120,40){\circle{1}}
\put(135,25){\circle{1}}
\put(135,25){\line(1,1){15}}
\put(120,40){\line(1,-1){15}}
\put(150,40){\circle{1}}
\put(190,40){\circle{1}}

\put(10,0){\circle{1}}

\put(50,0){\circle{1}}
\put(90,0){\circle{1}}
\put(100,0){\circle{1}}

\put(130,0){\circle{1}}
\put(140,0){\circle{1}}
\put(150,0){\circle{1}}
\put(190,0){\circle{1}}

\put(20,10){\circle{1}}
\put(120,10){\circle{1}}

\put(70,20){\circle*{1}}
\put(170,20){\circle*{1}}

\put(50,40){\circle{1}}
\put(90,40){\circle{1}}

\multiput(12,2)(1,1){40}{\line(1,0){0.2}}
\multiput(52,2)(1,1){40}{\line(1,0){0.2}}
\multiput(112,2)(1,1){40}{\line(1,0){0.2}}
\multiput(152,2)(1,1){20}{\line(1,0){0.2}}
\multiput(152,2)(1,1){40}{\line(1,0){0.2}}
\multiput(51,39)(1,-1){40}{\line(1,0){0.2}}
\multiput(91,39)(1,-1){40}{\line(1,0){0.2}}
\multiput(151,39)(1,-1){40}{\line(1,0){0.2}}
\multiput(101,39)(1,-1){40}{\line(1,0){0.2}}
\multiput(111,39)(1,-1){40}{\line(1,0){0.2}}
\multiput(60,40)(-1,-1){5}{\line(1,0){0.2}}
\multiput(60,40)(1,-1){5}{\line(1,0){0.2}}
\multiput(70,40)(-1,-1){5}{\line(1,0){0.2}}
\multiput(110,40)(-1,-1){5}{\line(1,0){0.2}}
\multiput(0,0)(1,0){200}{\line(1,0){0.2}}
\multiput(0,5)(1,0){200}{\line(1,0){0.2}}
\multiput(5,10)(1,0){200}{\line(1,0){0.2}}
\multiput(5,15)(1,0){200}{\line(1,0){0.2}}
\multiput(0,40)(1,0){200}{\line(1,0){0.2}}
\multiput(5,35)(1,0){200}{\line(1,0){0.2}}

\put(-10,-2){\tiny$b=1$}
\put(-10,3){\tiny$b=2$}
\put(-10,8){\tiny$b=n-h$}
\put(-10,13){\tiny$b=n-h+1$}
\put(-10,33){\tiny$b=n-1$}

\put(-10,38){\tiny$b=n$}

\put(5,-5){\tiny$a+h$}
\put(15,-5){\tiny$a+h+1$}
\put(64,41){\tiny$a+h+2$}
\put(73,44){\tiny$a+h+3$}
\put(50,-5){\tiny$a$}
\put(80,-3){\tiny$a+h-2$}
\put(92,-5){\tiny$a+h-1$}
\put(105,-3){\tiny$a+h$}
\put(125,-5){\tiny$a-2$}
\put(138,-5){\tiny$a-1$}
\put(150,-5){\tiny$a$}
\put(185,-5){\tiny$a+h-2$}

\put(96,42){\tiny$a+1$}
\put(106,42){\tiny$a+2$}
\put(116,42){\tiny$a+3$}
\put(132,19){\tiny$(a+h,n-h+2)$}
\put(72,19){\tiny$(a,h-1)$}
\put(172,19){\tiny$(a,h-1)$}
 \end{picture}
 \vspace{.15cm}
 \caption{An illustration of the $7$ subsets of indecomposable rigid objects $M$ with different colors.}
 \end{figure}
 
\begin{itemize}
\item[{\bf Case 1:}] $M\in \cw_{(a+h+1,n-1)}$. It is clear that $M$ and $\Sigma M$ belong to $\cw_{X^*}$ and $X^*=\Sigma^{-1}(a+h-1,n)$. By Lemma~\ref{l:dim} (2), we know that $\Hom_\cc(M\oplus \Sigma M, (a+h-1,n))=0$, 
which implies that
\[0=\Hom_\cc(M\oplus \Sigma M, (a+h-1,n))=
\Hom_\cc(M\oplus \Sigma M, \Sigma (a+h,n)).
\] Applying the functor $\Hom_{\cc}(M\oplus \Sigma M,-)$ to the triangle 
\[(a+h,n)\to (a,h-1)\oplus (a,h-1)\xrightarrow{(g_1,g_2)} (a,n)\to\Sigma(a+h,n),\]
 we conclude that each morphism from $M\oplus \Sigma M$ to $(a,n)$ factors through $(g_1,g_2)$.
 \item[{\bf Case 2:}] $M\in ~^\sqsupset(a+h-1,1)\cap (a+h+2,n-1)^{\sqsubset}$. By the Hom-hammock, it is not hard to see that $\Hom_\cc(M\oplus\Sigma M, (a+h,n))=0$. Consequently, each morphism from $M\oplus \Sigma M$ to $(a+h,n)$ factors through $(f_1,f_2): (a+h,n-h)\oplus (a+h,n-h)\xrightarrow{} (a+h,n)$.
 
 \item[{\bf Case 3:}] $M\in ~^\sqsupset(a+h,1)\cap (a+h+2,n)^{\sqsubset}$. It is clear that $\Hom_\cc(\Sigma M, (a+h,n))=0$ in this case. It remains to show that each morphism from $M$ to $(a+h,n)$ factors through $(f_1,f_2): (a+h,a-h)\oplus (a+h,a-h)\xrightarrow{} (a+h,n)$.
 Clearly we have $(a+h,n-h)\in M^{\sqsubset}\cap~^\sqsupset(a+h,n)$. According to Lemma~\ref{l:factor-morphism} (iii), each $\ct$-map from $M$ to $(a+h,n)$ factors through $f_1: (a+h,n-h)\xrightarrow{} (a+h,n)$.  By Lemma~\ref{l:factor-morphism} (iv), each $\der$-map from $M$ to $(a+h,n)$ factors through $f_2: (a+h,n-h)\xrightarrow{} (a+h,n)$. Thus each morphism from $M$ to $(a+h,n)$ factors through $(f_1,f_2)$.
  
 \item[{\bf Case 4:}] $M\in ~^\sqsupset(a+h,n-h)\backslash ~^\sqsupset(a+h,1)$. It follows from Lemma~\ref{l:factor-morphism} (iii) and (iv) that each morphism from $M$ to $(a+h,n)$ factors through $(f_1,f_2)$. Note that $\Sigma M$ belongs to either $~^\sqsupset(a+h,n-h)\backslash ~^\sqsupset(a+h,1)$ or $~^\sqsupset(a+h,1)\cap (a+h+2,n)^{\sqsubset}$. Thus the aforementioned proof applied and we conclude that each morphism from $\Sigma M$ to $(a+h,n)$ also factors through $(f_1,f_2)$.
 
 \item[{\bf Case 5:}] $M\in ~^\sqsupset(a+h,n-h+1)\cap (a+2,n-1)^\sqsubset$. 
 The claim follows from the fact that  $\Hom_\cc(M\oplus \Sigma M, (a,n))=0$ indicated by Lemma~\ref{l:dim} (2).
 
 \item[{\bf Case 6:}] $M\in~^\sqsupset(a+h,n-h+2)\cap (a+2,n)^{\sqsubset}$. According to Lemma~\ref{l:factor-morphism} (iii) and (iv), we know that each morphism from $M$ to $(a,n)$ factors through $(g_1,g_2)$. Then again the claim follows from the fact that  $\Hom_\cc(\Sigma M, (a,n))=0$ indicated by Lemma~\ref{l:dim} (2). 
 
 \item[{\bf Case 7:}] $M\in (a+h,n)^\sqsupset\backslash (a+h,n-h+2)^\sqsupset$.
The claim follows from Lemma~\ref{l:factor-morphism} (iii) and (iv).
\end{itemize}
\end{proof}

\begin{proposition}~\label{p:exchange}
Suppose that $X=(a,b)$ with $b<n$, and $(X,X^*)$ is an exchange pair. Then any indecomposable rigid object $M$ in $\cc$ is compatible with $(X,X^*)$.
\end{proposition}
\begin{proof}
By Lemma \ref{l:exchange triangle in cluster tube}, we may assume that  $X^*=(a+h,b-h+i)$ and  the exchange triangles are of the following forms:
\[(a,b)\ra (a,b+i)\oplus(a+h,b-h)\xrightarrow{(f_1,f_2)} (a+h,b-h+i)\ra \Sigma (a,b)\]
\[(a+h,b-h+i)\ra (a+b+1,i-1)\oplus(a,h-1)\xrightarrow{(g_1,g_2)} (a,b)\ra \Sigma (a+h,b-h+i),
\]
where $1\leq a\leq n+1, 1\leq b<n, 1\leq h\leq b, 1\leq i\leq n-b$. Without loss of generality, we may assume that $f_1,f_2,g_2$ are $\ct$-maps and $g_1$ is a $\der$-map. 
The result is clear for $M\cong \Sigma^{-1}X$ or $M\cong \Sigma^{-1}X^*$. In the following, we assume that $M\not\cong \Sigma^{-1}X$ and $M\not\cong \Sigma^{-1}X^*$.
According to Lemma~\ref{l:judge for exchangeable}, it suffices to show that either each non-zero morphism from $M\oplus \Sigma M$ to $X^*$ factors through $(f_1,f_2): (a,b+i)\oplus (a+h,b-h)\xrightarrow{} X^*$ or each non-zero morphism from $M\oplus \Sigma M$ to $X$ factors through $(g_1,g_2): (a+b+1,i-1)\oplus (a,h-1)\xrightarrow{}(a,b)$.

Let us first consider that $M\not\in (\Sigma^{-1}X)^{\sqsubset}\cap~^\sqsupset(\Sigma^{-1}X^*)$. Let $f: M\to X^*$ be a non-zero $\ct$-map. Consequently, $M\in~^\sqsupset X^*$.
It is not hard to see that either $B_1:=(a,b+i)\in M^\sqsubset\cap~^\sqsupset X^*$ or $B_2:=(a+h,b-h)\in M^\sqsubset\cap~^\sqsupset X^*$. By Lemma~\ref{l:factor-morphism} (iii), we deduce that $f$ factors through $(f_1,f_2): (a,b+i)\oplus (a+h,b-h)\xrightarrow{} X^*$. Now assume that $g:M\to X^*$ is a non-zero $\der$-map, then $M\in (\Sigma^{-2}{X^*})^{\sqsubset}$. Note that $(\Sigma^{-2}{X^*})^{\sqsubset}\subset (\Sigma^{-2}B_1)^{\sqsubset}$. In particular by Lemma~\ref{l:factor-morphism} (i)(iii), the morphism $f_1$ yields an isomorphism of vector spaces
\[\Hom_\ct(\Sigma^{-2}f_1,M):\Hom_\ct(\Sigma^{-2}X^*, M)\xrightarrow{\sim} \Hom_\ct(\Sigma^{-2}B_1, M)
\]
 By the AR-duality, we have an isomorphism
\[\Hom_\ct(M,\tau^{-1}\circ\Sigma f_1):\Hom_\ct(M,\tau^{-1}\circ\Sigma B_1)\xrightarrow{\sim} \Hom_\ct(M,\tau^{-1}\circ\Sigma X^*),
\]
which implies that $g$ factors through $(f_1,f_2): (a,b+i)\oplus (a+h,b-h)\xrightarrow{} X^*$. In particular, we have proved that  each non-zero morphism from $M$ to  $X^*$ factors through the morphism $(f_1, f_2)$ for $M\not\in (\Sigma^{-1}X)^{\sqsubset}\cap~^\sqsupset(\Sigma^{-1}X^*)$. On the other hand, for an indecomposable rigid object $M\not\in (\Sigma^{-1}X^{\sqsubset})\cap~^\sqsupset(\Sigma^{-1}X^*)$, we have either $\Sigma M\not\in (\Sigma^{-1}X)^{\sqsubset}\cap~^\sqsupset(\Sigma^{-1}X^*)$ or $\Sigma M\in ((\Sigma^{-1}B_1)^{\sqsubset}\cap ^{\sqsupset}(\Sigma^{-1} X^*))\cup ((\Sigma^{-1}B_2)^{\sqsubset}\cap ^{\sqsupset}(\Sigma^{-1} X^*))=:\mathcal{S}$.   If $\Sigma M\not\in (\Sigma^{-1}X)^{\sqsubset}\cap~^\sqsupset(\Sigma^{-1}X^*)$, then the above discussion implies that each non-zero morphism from $\Sigma M$ to $X^*$ factors through the morphism $(f_1, f_2)$. If $\Sigma M\in \mathcal{S}$, it is easy to see that $\Hom_\cc(\Sigma M, X^*)=0$ by the $\Hom$-hammock. Hence we conclude that $M$ is compatible with the exchange pair $(X, X^*)$ whenever $M\not\in (\Sigma^{-1}X)^{\sqsubset}\cap~^\sqsupset(\Sigma^{-1}X^*)$.

It remains to consider the situation $M\in (\Sigma^{-1}X)^{\sqsubset}\cap~^\sqsupset(\Sigma^{-1}X^*)$ and we will divide the proof into $4$ cases ({\it cf.} Figure 3.).

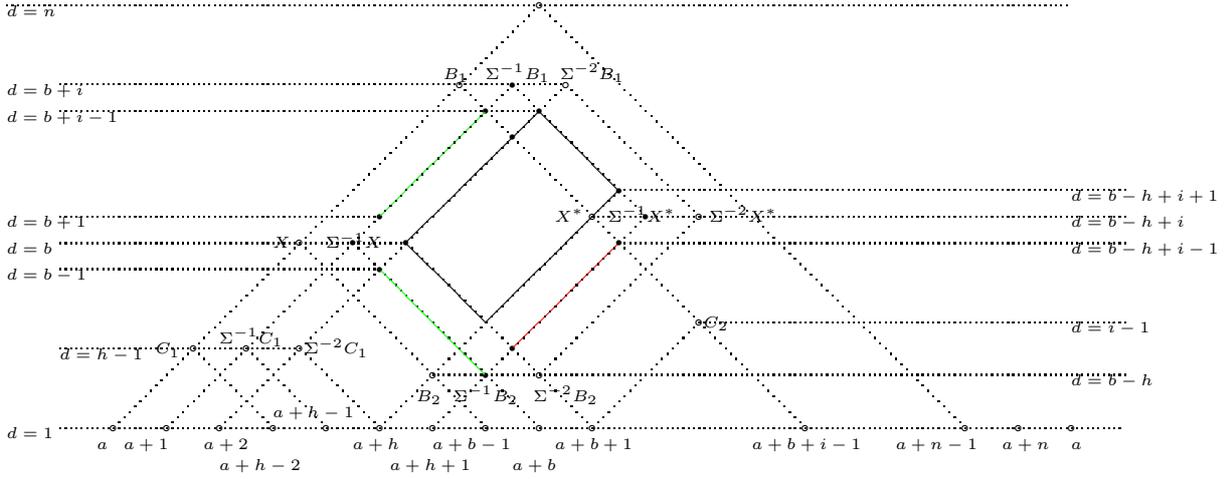
\begin{figure}
\setlength{\unitlength}{0.070cm}
\begin{picture}(290,90)
\put(15,0){\circle{1}}
\put(25,0){\circle{1}}
\put(35,0){\circle{1}}
\put(65,0){\circle{1}}
\put(45,0){\circle{1}}
\put(55,0){\circle{1}}
\put(105,0){\circle{1}}
\put(95,0){\circle{1}}
\put(85,0){\circle{1}}
\put(145,0){\circle{1}}
\put(175,0){\circle{1}}
\put(185,0){\circle{1}}
\put(195,0){\circle{1}}
\put(75,0){\circle{1}}
\multiput(75,0)(1,1){15}{\line(1,0){0.2}}

\put(30,15){\circle{1}}
\put(40,15){\circle{1}}
\put(50,15){\circle{1}}

\put(75,10){\circle{1}}
\put(85,10){\circle*{1}}
\put(95,10){\circle{1}}
\put(90,15){\circle*{1}}

\put(125,20){\circle{1}}

\put(80,65){\circle{1}}
\put(90,65){\circle*{1}}
\put(100,65){\circle{1}}
\put(95,60){\circle*{1}}
\put(85,60){\circle*{1}}
\put(90,55){\circle*{1}}

\put(50,35){\circle{1}}
\put(60,35){\circle*{1}}
\put(70,35){\circle*{1}}
\multiput(60,35)(1,1){30}{\line(1,0){0.2}}
\multiput(60,35)(1,-1){35}{\line(1,0){0.2}}
\multiput(70,35)(1,1){30}{\line(1,0){0.2}}
\multiput(70,35)(1,-1){35}{\line(1,0){0.2}}
\put(65,40){\circle*{1}}
\put(65,30){\circle*{1}}
{\color{green}\put(65,40){\line(1,1){20}}
\put(65,30){\line(1,-1){20}}}
\put(70,35){\line(1,1){25}}
\put(70,35){\line(1,-1){15}}
\put(85,20){\line(1,1){25}}
{\color{red}\put(90,15){\line(1,1){20}}}
\put(110,45){\line(-1,1){15}}

\put(105,40){\circle{1}}
\put(110,45){\circle*{1}}
\put(110,35){\circle*{1}}
\put(115,40){\circle*{1}}
\put(125,40){\circle{1}}
\multiput(110,45)(-1,-1){5}{\line(1,0){0.2}}
\multiput(115,40)(-1,-1){30}{\line(1,0){0.2}}
\multiput(115,40)(-1,1){25}{\line(1,0){0.2}}
\multiput(125,40)(-1,-1){30}{\line(1,0){0.2}}
\multiput(125,40)(-1,1){25}{\line(1,0){0.2}}

\put(95,80){\circle{1}}

\multiput(65,0)(1,1){40}{\line(1,0){0.2}}
\multiput(80,65)(1,-1){65}{\line(1,0){0.2}}
\multiput(95,80)(1,-1){80}{\line(1,0){0.2}}
\multiput(15,0)(1,1){80}{\line(1,0){0.2}}
\multiput(25,0)(1,1){65}{\line(1,0){0.2}}
\multiput(35,0)(1,1){65}{\line(1,0){0.2}}
\multiput(105,0)(1,1){20}{\line(1,0){0.2}}
\multiput(30,15)(1,-1){15}{\line(1,0){0.2}}
\multiput(85,0)(-1,1){35}{\line(1,0){0.2}}
\multiput(65,0)(-1,1){15}{\line(1,0){0.2}}
\multiput(55,0)(-1,1){15}{\line(1,0){0.2}}

\multiput(5,0)(1,0){200}{\line(1,0){0.2}}
\put(-5,-2){\tiny$d=1$}
\multiput(-5,80)(1,0){200}{\line(1,0){0.2}}
\put(-5,78){\tiny$d=n$}
\multiput(5,35)(1,0){65}{\line(1,0){0.2}}
\put(-5,33){\tiny$d=b$}
\multiput(5,30)(1,0){60}{\line(1,0){0.2}}
\put(-5,28){\tiny$d=b-1$}
\multiput(5,40)(1,0){60}{\line(1,0){0.2}}
\put(-5,38){\tiny$d=b+1$}
\multiput(5,15)(1,0){45}{\line(1,0){0.2}}
\put(5,13){\tiny$d=h-1$}
\multiput(205,40)(-1,0){100}{\line(1,0){0.2}}
\put(195,38){\tiny$d=b-h+i$}
\multiput(205,35)(-1,0){95}{\line(1,0){0.2}}
\put(195,33){\tiny$d=b-h+i-1$}
\multiput(205,45)(-1,0){95}{\line(1,0){0.2}}
\put(195,43){\tiny$d=b-h+i+1$}
\multiput(205,10)(-1,0){130}{\line(1,0){0.2}}
\put(195,8){\tiny$d=b-h$}

\multiput(205,20)(-1,0){80}{\line(1,0){0.2}}
\put(195,18){\tiny$d=i-1$}

\multiput(5,65)(1,0){95}{\line(1,0){0.2}}
\put(-5,63){\tiny$d=b+i$}
\multiput(5,60)(1,0){90}{\line(1,0){0.2}}
\put(-5,58){\tiny$d=b+i-1$}
\put(45,34){\tiny$X$}
\put(55,34){\tiny$\Sigma^{-1}X$}
\put(98,39){\tiny$X^*$}
\put(108,39){\tiny$\Sigma^{-1}X^*$}
\put(127,39){\tiny$\Sigma^{-2}X^*$}
\put(77,66){\tiny$B_1$}
\put(85,66){\tiny$\Sigma^{-1}B_1$}
\put(99,66){\tiny$\Sigma^{-2}B_1$}
\put(72,5){\tiny$B_2$}
\put(79,5){\tiny$\Sigma^{-1}B_2$}
\put(94,5){\tiny$\Sigma^{-2}B_2$}
\put(23,14){\tiny$C_1$}
\put(35,16){\tiny$\Sigma^{-1}C_1$}
\put(51,14){\tiny$\Sigma^{-2}C_1$}
\put(126,19){\tiny$C_2$}

\put(12,-4){\tiny$a$}
\put(17,-4){\tiny$a+1$}
\put(32,-4){\tiny$a+2$}
\put(35,-8){\tiny$a+h-2$}
\put(60,-4){\tiny$a+h$}
\put(67,-8){\tiny$a+h+1$}
\put(75,-4){\tiny$a+b-1$}
\put(98,-4){\tiny$a+b+1$}
\put(135,-4){\tiny$a+b+i-1$}
\put(162,-4){\tiny$a+n-1$}
\put(182,-4){\tiny$a+n$}
\put(195,-4){\tiny$a$}
\put(45,2){\tiny$a+h-1$}
\put(90,-8){\tiny$a+b$}
\end{picture}
\vspace{.15cm}
\caption{An illustration of the $4$ subsets of indecomposable rigid objects $M$ with different colors, where $B_1= (a,b+i),B_2=(a+h,b-h),C_1=(a,h-1),C_2= (a+b+1,i-1).$}
\end{figure}
\begin{itemize}
\item[{\bf Case 1:}] $M\in (a+1,b+1)^{\sqsubset}\cap~^\sqsupset(a+1,b+i-1)$ or $(a+2,b-1)^{\sqsubset}\cap~^{\sqsupset}(a+h+1,b-h)$. Since $b+i\leq n$, it is not hard to see that $\Hom_\cc(M\oplus \Sigma M, X)=0$.
Consequently, $M$ is compatible with $(X,X^*)$ in this case.

\item[{\bf Case 2:}] $M\in (\Sigma^{-2}X)^{\sqsubset}\cap~^{\sqsupset}(a+h,b-h+i+1)$. Again by $b+i\leq n$, we deduce that $\Hom_\ct(M,X)=0$. On the other hand, we have $M\in (a+2,h-1)^{\sqsubset}$. By Lemma~\ref{l:factor-morphism} (i)(iii) and the AR-duality, the morphism $g_2$ also induces an isomorphism
\[\Hom_{\der^b(\ct)}(M,g_2): \Hom_{\der^b(\ct)}(M,\tau^{-1}\circ \Sigma (a,h-1))\xrightarrow{\sim}\Hom_{\der^b(\ct)}(M,\tau^{-1}\circ\Sigma X).
\]
Consequently, each $\der$-map from $M$ to $X$ factors through $(g_1,g_2)$.
Note that $\Sigma M$ either is isomorphic to $\Sigma^{-1}X$ or belongs to $(a+1,b+1)^{\sqsubset}\cap~^\sqsupset(a+1,b+i-1)$, $(a+2,b-1)^{\sqsubset}\cap~^{\sqsupset}(a+h+1,b-h)$ or $(\Sigma^{-2}X)^{\sqsubset}\cap~^{\sqsupset}(a+h,b-h+i+1)$. If $\Sigma M\cong \Sigma^{-1}X$, then $\Hom_\cc(\Sigma M, X)=0$. Therefore each morphism from $M\oplus \Sigma M$ to $X$ factors through $(g_1,g_2)$ in this situation. If $\Sigma M$ belongs to one of the three sets $(a+1,b+1)^{\sqsubset}\cap~^\sqsupset(a+1,b+i-1)$, $(a+2,b-1)^{\sqsubset}\cap~^{\sqsupset}(a+h+1,b-h)$ and $(\Sigma^{-2}X)^{\sqsubset}\cap~^{\sqsupset}(a+h,b-h+i+1)$, then by the aforementioned proof we know that each morphism from $\Sigma M$ to $X$ also factors through $(g_1,g_2)$. Hence $M$ is compatible with $(X, X^*)$.

\item[{\bf Case 3:}] $M\in (a+h+1,b-h+1)^\sqsubset\cap~^{\sqsupset}(a+h+1, b-h+i-1)$. It is clear that $\Hom_\ct(M, X)=0$. We also have $M\in (\Sigma^{-2}X)^\sqsubset\cap~^{\sqsupset}(a+b+1,i-1)$.
Now by Lemma~\ref{l:factor-morphism} (iv), each $\der$-map from $M$ to $X$ factors through $g_1:(a+b+1,i-1)\to X$.
On the other hand, the object $\Sigma M$ either belongs to 
$(a+2,b-1)^{\sqsubset}\cap~^{\sqsupset}(a+h+1,b-h)$ or $(\Sigma^{-2}X)^{\sqsubset}\cap~^{\sqsupset}(a+h,b-h+i+1)$. Thus by the proof of Case 1 and Case 2, we deduce that each morphism from $\Sigma M$ to $X$ factors through $(g_1,g_2)$.

\item[{\bf Case 4:}]$M\cong \Sigma^{-1}(a,b+i)$. By $b<n$, one can easily show that $\Hom_\cc(M\oplus \Sigma M, X)=0$.
\end{itemize}
This completes the proof.
\end{proof}

Combining Proposition~\ref{p:exchange-n} with \ref{p:exchange}, we have proved the following result.
\begin{theorem}~\label{t:exchange compatible}
Each indecomposable rigid object in $\cc$ is exchange compatible.
\end{theorem}

\subsection{Rank vectors of indecomposable $\tau$-rigid modules}~\label{ss:rank-vector}

Let $T=\overline{T}\oplus T_k$ be a basic maximal rigid object of $\cc$ with indecomposable direct summand $T_k$. Let $T'=\overline{T}\oplus T_k^*$ be the mutation of $T$ at the indecomposable direct summand $T_k$. Denote by $\Gamma=\End_{\cc}(T)$ and $\Gamma'=\End_{\cc}(T')$ the endomorphism algebras of $T$ and $T'$ respectively. The following lemma is a reformulation of Proposition~4.5 of ~\cite{FuGeng}.
\begin{lemma}~\label{l:dimension-vector-vs-mutation}
 If different indecomposable $\tau$-rigid $\Gamma'$-modules have different dimension vectors, then different indecomposable $\tau$-rigid $\Gamma$-modules have different dimension vectors.
\end{lemma}
\begin{proof}
Let $T_k\xrightarrow{f}B\xrightarrow{g}T_k^*\to \Sigma T_K$ and $T_k^*\xrightarrow{f'}B'\xrightarrow{g'}T_k\to\Sigma T_k^*$ be the exchange triangles associated to the exchange pair $(T_k,T_k^*)$.
Suppose that there are two non-isomorphic indecomposable rigid objects $M,N\in \cc\backslash \add \Sigma T$ such that $\dimv \Hom_\cc(T,M)=\dimv \Hom_\cc(T,N)$. We claim that $M\not\cong \Sigma T_k^*$ and $N\not\cong \Sigma T_k^*$. Otherwise, assume that $M\cong \Sigma T_k^*$. According to Lemma~\ref{l:exchange triangle in cluster tube}, we conclude that 
\begin{itemize}
\item[$\circ$]$\Hom_\cc(T,M)$ is the simple $\Gamma$-module with dimension vector $e_k$ provided $l(T_k)<n$;
\item[$\circ$] $\Hom_\cc(T,M)$ is the indecomposable $\Gamma$-module with dimension vector $2e_k$ provided $l(T_k)=n$. 
\end{itemize}
In either case, we have $\Hom_{\cc}(T,N)\cong \Hom_{\cc}(T,M)$ as $\Gamma$-modules, a contradiction.

Since $M,N\not\in \add \Sigma T$ and $M, N$ are exchange compatible by Theorem~\ref{t:exchange compatible}, we have
\[\dim_K\Hom_\cc(T_k, M)+\dim_K\Hom_\cc(T_k^*,M)=\max\{\dim_K\Hom_\cc(B,M),\dim_K\Hom_\cc(B',M)\}
\]
and
\[\dim_K\Hom_\cc(T_k, N)+\dim_K\Hom_\cc(T_k^*,N)=\max\{\dim_K\Hom_\cc(B,N),\dim_K\Hom_\cc(B',N)\}
\]
by Lemma~\ref{l:compatible}.
Note that $B,B'\in \add \overline{T}$ and $\dimv \Hom_\cc(T, M)=\dimv \Hom_\cc(T,N)$, we deduce that 
\[\dim_K\Hom_\cc(B,M)=\dim_K\Hom_\cc(B,N)~\text{and}~\dim_K\Hom_\cc(B',M)=\dim_K\Hom_\cc(B',N).
\]
Consequently, $\dim_K\Hom_\cc(T_k^*,M)=\dim_K\Hom_\cc(T_k^*,N)$, which implies that $\dimv \Hom_\cc(T', M)=\dimv\Hom_\cc(T',N)$, a contradiction. Therefore different indecomposable $\tau$-rigid $\Gamma$-modules have different dimension vectors.
\end{proof}

Let $T=\bigoplus\limits_{i=1}^nT_i$ be a basic maximal rigid object in $\cc$ and $\Gamma=\End_{\cc}(T)$ the endomorphism algebra of $T$. Let $Q_T$ be the Gabriel quiver of $\Gamma$ with vertex set $Q_0=\{1,2,\cdots,n\}$. 
 Denote by $f_1, \cdots, f_n$ the primitive idempotents of $\Gamma$ associated to the vertices. Following ~\cite{GLS14}, an $\Gamma$-module $M$ is {\it locally free} if for each $f_i$, $Mf_i$ is free as a right $f_i\Gamma f_i$-module. For a locally free $\Gamma$-module $M$, let $r(Mf_i)$ be the rank of $Mf_i$ as a free $f_i\Gamma f_i$-module. We call $\rankv~M:=(r(Mf_1),\cdots, r(Mf_n))^{\opname{tr}}\in \Z^n$
the {\it rank vector} of $M$. Recall that $Q_T$ admits a unique vertex, say $1$, such that there is a unique loop attached to the vertex $1$. In this case, if $M$ is a locally free $\Gamma$-module with $\dimv M=(m_1,\cdots, m_n)^{\opname{tr}}$, then $\rankv~M=(\frac{m_1}{2}, m_2,\cdots, m_n)^{\opname{tr}}$.

\begin{lemma}
Each indecomposable $\tau$-rigid $\Gamma$-module is locally free.
\end{lemma}
\begin{proof}
Let $T_1$ be the unique indecomposable direct summand of $T$ with length $l(T_1)=n$ and $P_1$ be the projective $\Gamma$-module corresponding to $T_1$. It follows from Theorem \ref{t:algbra-maximal-tube} and the construction of string modules \cite{ButlerRingel} of $\Gamma$ that an indecomposable $\Gamma$-module $M$ is locally free if and only if $\dim_K\Hom_\Gamma(P_1, M)=0$ or $2$.  Let $X\not\in \add \Sigma T$ be the indecomposable rigid object with $l(X)\leq n $ corresponding to $M$, \ie $M=FX$, where $F$ is the equivalence $
 \opname{pr} T/\add \Sigma T\xrightarrow{\sim} \mod \Gamma.
 $ Then in light of Lemma \ref{l:dim}, $\dim_K\Hom_\Gamma(P_1, M)=\dim_K\Hom_\cc(T_1, X)=0$ or $2$. 
 \end{proof}

We now prove that indecomposable $\tau$-rigid $\Gamma$-modules are determined by their rank vectors.
\begin{theorem}~\label{t:rank-vector-of-tau-rigid}
Let $T$ be a basic maximal rigid object of $\cc$ and $\Gamma=\End_\cc(T)$ the endomorphism algebra of $T$. Then different indecomposable $\tau$-rigid $\Gamma$-modules have different rank vectors.
\end{theorem}
\begin{proof}
It suffices to prove that different indecomposable $\tau$-rigid $\Gamma$-modules have different dimension vectors.
Let $T_1$ be the unique indecomposable direct summand of $T$ with length $l(T_1)=n$. Without loss of generality, we may assume that $T_1=(1,n)$.  Denote by $T'=(1,1)\oplus (1,2)\oplus \cdots\oplus (1,n)$. It is clear that $T'$ is also a basic maximal rigid object of $\cc$. Moreover, $T$ can be obtained from $T'$ by a sequence of mutations ({\it cf.} Lemma 3.3 of~\cite{Yang12}). According to Theorem~1.3 of~\cite{GLS14}, we know that different indecomposable $\tau$-rigid $\End_\cc(T')$-modules  have different dimension vectors. Applying Lemma~\ref{l:dimension-vector-vs-mutation} repeatedly, we conclude that different indecomposable $\tau$-rigid $\Gamma$-modules have different dimension vectors.
\end{proof}

\section{The denominator theorem and its applications}
\subsection{The denominator theorem}~\label{ss:denominator-thm}
Let $T=\bigoplus\limits_{i=1}^n T_i$ be a basic maximal rigid object of $\cc$ with $l(T_1)=n$ and $\mathcal{A}_T=\mathcal{A}(B_T)$ the associated cluster algebra of type $\mathrm{C}_n$. Recall that we have a bijection $\mathbb{X}_?$ between the  indecomposable rigid objects of $\cc$ and the cluster variables of $\mathcal{A}_T$ by Theorem~\ref{t:BMV-main-theorem} such that $\Sigma T$ corresponds to the initial cluster of $\mathcal{A}_T$. In particular, $\mathbb{X}_{\Sigma T_i}=x_i$ for $i=1,\cdots, n$.
Let $\Gamma=\End_\cc(T)$ be the endomorphism algebra of $T$. Recall that $\operatorname{pr} T$ is the subcategory of $\cc$ consisting of objects which are finitely presented by $T$. The functor $F:=\Hom_\cc(T,?)$ yields an equivalence 
\[F: \operatorname{pr} T/\add \Sigma T\to \mod \Gamma
\]
 which maps a rigid object of $\cc$ to a locally free $\Gamma$-module.
The aim of this section is to prove that for each indecomposable rigid object $M\not\in \add \Sigma T$, the denominator vector of the  cluster variable $\mathbb{X}_M$ is the rank vector of  the locally free $\Gamma$-module $F(M)$.

 For an object $M\in \operatorname{pr} T$, we 
define a monomial $t_M:=\prod_{i=1}^nx_i^{d_i(M)}$, where 
\[d_i(M)=\begin{cases}\frac{\dim_K\Hom_\cc(T_i, M)}{2}& i=1;\\ \dim_K\Hom_\cc(T_i, M)& \text{else.}\end{cases}
\]

Let $f(x_1,\cdots, x_n)$ be a polynomial in variables $x_1,\cdots, x_n$ with integer coefficients. We say that $f(x_1,\cdots, x_n)$  satisfies the {\it positive condition}  if  $f(\epsilon_i)>0$ where $\epsilon_i:=(1,\cdots,1,0,1,\cdots 1)\in \Z^n$(with a $0$ in the $i$-th position) for $i=1,\cdots, n$.

Let $M$ be an indecomposable rigid object of $\cc$. Following~\cite{BMR}, the cluster variable $\mathbb{X}_M$ has a {\it $T$-denominator} if $\mathbb{X}_M=\begin{cases}\frac{f(x_1,\cdots, x_n)}{t_M}& M\not\in \add \Sigma T\\ x_i&M\cong \Sigma T_i\end{cases}$, where $f(x_1,\cdots, x_n)$  is a polynomial satisfying the positive condition. In this case, by the positive condition, we clearly know that the polynomial $f(x_1,\cdots, x_n)$ is not divisible by $x_i$ for each $i=1,\cdots,n$. In particular, the denominator vector $\opname{den} (\mathbb{X}_M)$ coincides with the rank vector $\opname{\underline{rank}} F(M)$.

\begin{lemma}~\label{l:exchange-dimension}
Let $(M,M^*)$ be  an exchange pair with exchange triangles
\[M\to B\to M^*\to \Sigma M~\text{and}~ M^*\to B'\to M\to \Sigma M^*.
\]
Let $N$ be an indecomposable rigid object of $\cc$ and suppose that either $M \cong \tau N$ or $M^*\cong \tau N$. Then we have 
\begin{eqnarray*}
&&\dim_K\Hom_\cc(N,M)+\dim_K\Hom_\cc(N,M^*)\\
&=&\begin{cases}
\max\{\dim_K\Hom_\cc(N,B),\dim_K\Hom_\cc(N,B')\}+1& l(N)<n;\\
\max\{\dim_K\Hom_\cc(N,B), \dim_K\Hom_\cc(N,B')\}+2& l(N)=n.
\end{cases}
\end{eqnarray*}
\end{lemma}
\begin{proof} Assume that $M^*\cong \tau N$ (the other case is similar) and we can rewrite the  exchange triangles as
\[M\ra B\ra \tau N\ra\Sigma M~\text{ and}~ \tau N\ra B'\ra M\ra\Sigma N.
\]
  Note that  $\Hom_\cc(N,M^*)=0$, we have 
  \[\dim_K\Hom_\cc(N,M) + \dim_K \Hom_\cc(N,M^*)=\dim_K \Hom_\cc(N,M) =\dim_K D\Ext^1_{\cc}(M,\tau N).
  \]
    By Lemma \ref{l:exchange triangle in cluster tube}, 
    we know that $\dim_K\Ext^1_\cc(M,M^*)=\begin{cases}1& l(M)<n;\\ 2&l(M)=n.\end{cases}$.
    Consequently, 
    \[\dim_K \Hom_\cc(N,M) + \dim_K \Hom_\cc(N,M^*) = \begin{cases}1&l(N)<n;\\2&l(N)=n.\end{cases}
    \] 
    On the other hand, 
    \[\Hom_\cc(N,B')\cong \Hom_\cc(M^*,\tau B')\cong D\Ext^1_{\cc}(B', M^*)=0,\]
    for the reason that $B'\oplus M^*$ is a direct summand of a maximal rigid object. Similarly $\Hom_\cc(N,B)=0$ and we are done.
\end{proof}

The following lemma is an analogue of Proposition 3.1 of ~\cite{BMR}. 
\begin{lemma}~\label{l:t-denominator}
Let $M=\overline{M}\oplus M_k$ be a basic maximal rigid object of $\cc$ with an indecomposable direct summand $M_k$. Let $\mu_{M_k}(M)=\overline{M}\oplus M_k^*$ be the mutation of $M$ at  $M_k$.  If for each indecomposable direct summand $N$ of $M$, the cluster variable $\mathbb{X}_{N}$ has  a $T$-denominator, then the cluster variable $\mathbb{X}_{M_k^*}$ also has a $T$-denominator.
\end{lemma}

\begin{proof}
Let us rewrite $M=\bigoplus\limits_{i=1}^nM_i$, where $M_1,\cdots, M_n$ are indecomposable objects.
By the assumption, for each indecomposable direct summand $M_i$ such that $M\not\in \Sigma T$, we have $\mathbb{X}_{M_i}=\frac{f_{M_i}}{t_{M_i}}$, where $f_{M_i}$ is a polynomial in $x_1,\cdots, x_n$ satisfying the positive condition. We also set $f_{\Sigma T_j}=1$ for $1\leq j\leq n$. For an object $N$ with decomposition $N=N_1\oplus \cdots\oplus N_s$ of indecomposable direct summands, we set
\[\mathbb{X}_N=\mathbb{X}_{N_1}\mathbb{X}_{N_2}\cdots\mathbb{X}_{N_s}.
\]
Moreover, if $N\in \add M$, we also write
\[f_N=f_{N_1}f_{N_2}\cdots f_{N_s}.
\]

Let $M_k^*\to B\to M_k\to \Sigma M_k^*$ and $M_k\to B'\to M_k^*\to \Sigma M_k$ be the exchange triangles associated to the exchange pair $(M_k,M_k^*)$. In particular, $B,B'\in \add \overline{M}$.
We may rewrite $B=B_0\oplus B_1$ and $B'=B_0'\oplus B_1'$ such that $B_1,B_1'\in \add \Sigma T$ and $B_0, B_0'$ do not have indecomposable direct summands in $\add \Sigma T$. Consequently, $t_B=t_{B_0}$ and $t_{B'}=t_{B_0'}$.

  By definition of the skew-symmetrizable matrix $B_M$ and the exchange relation in the cluster algebra $\mathcal{A}_T$, we have
\[\mathbb{X}_{M_k}\mathbb{X}_{M_k^*}=\mathbb{X}_B+\mathbb{X}_{B'}.
\]
If $M_k^*\cong \Sigma T_i$ for some $i$, then we have $\mathbb{X}_{M_k^*}=x_i$ by Theorem~\ref{t:BMV-main-theorem}.
We separate the remaining proof into two cases.

\noindent{\bf Case 1:} Neither $M_k$ nor $M_k^*$ belongs to $\add\Sigma T$. By the assumption, we have  \[\mathbb{X}_B=\frac{f_B\mathbb{X}_{B_1}}{t_{B_0}}~\text{and}~\mathbb{X}_{B'}=\frac{f_{B'}\mathbb{X}_{B_1'}}{t_{B_0'}}.
\]
Set $m=\frac{\lcm(t_{B_0}, t_{B_0'})}{t_{B_0}}$ and $m'=\frac{\lcm(t_{B_0},t_{B_0'})}{t_{B_0'}}$, where $\lcm(t_{B_0}, t_{B_0'})$ is the least common multiple of $t_{B_0}$ and $t_{B_0'}$. Using the exchange relation, we have
\[\mathbb{X}_{M_k^*}=\frac{f_B\mathbb{X}_{B_1}/{t_{B_0}}+f_{B'}\mathbb{X}_{B_1'}/{t_{B_0'}}}{f_{M_k}/{t_{M_k}}}=\frac{(f_Bm\mathbb{X}_{B_1}+f_{B'}m'\mathbb{X}_{B_1'})/f_{M_k}}{\lcm(t_B,t_B')/t_{M_k}}.
\]
By Theorem~\ref{t:exchange compatible}, each indecomposable direct summand $T_i$ of $T$ is compatible with $(M_k, M_k^*)$. Therefore
\begin{eqnarray*}
t_{M_k}t_{M_k^*}&=&x_1^{\frac{\dim_K\Hom_\cc(T_1,M_k)}{2}+\frac{\dim_K\Hom_\cc(T_1, M_k^*)}{2}}\prod_{i=2}^nx_i^{\dim_K\Hom_\cc(T_i, M)+\dim_K\Hom_\cc(T_i, M_k^*)}\\
&=& x_1^{\max\{\frac{\dim_K\Hom_\cc(T_1,B)}{2}, \frac{\dim_K\Hom_\cc(T_1,B')}{2}\}}\prod_{i=2}^nx_i^{\max\{\dim_K\Hom_\cc(T_i,B), \dim_K\Hom_\cc(T_i,M_k^*)\}}\\
&=&\lcm(t_B,t_B').
\end{eqnarray*}
Consequently, 
\[\mathbb{X}_{M_k^*}=\frac{(f_Bm\mathbb{X}_{B_1}+f_{B'}m'\mathbb{X}_{B_1'})/f_{M_k}}{t_{M_k^*}}.
\]
 Since $B$ and $B'$ have no common factors, we know that $\mathbb{X}_{B_1}$ and $\mathbb{X}_{B_1'}$ are coprime. Suppose that $m$ and $\mathbb{X}_{B_1'}$ have a common factor $x_i$. By definition of $m$, there is a direct summand $X$ of $B_0'$ such that $\Hom_\cc(T_i, X)\neq 0$. On the other hand, $\Sigma T_i$ is a direct summand of $B_1'$. Consequently, $X$ and $\Sigma T_i$ are direct summands of the  basic maximal rigid object $M$. However, 
\[\Hom_\cc(X, \Sigma^2T_i)\cong D\Hom_\cc(T_i, X)\neq 0,
\]
which is a contradiction. Hence $m$ and $\mathbb{X}_{B_1'}$ are coprime. Similarly,  $m'$ and  $\mathbb{X}_{B_1}$ are also coprime. It follows that $m\mathbb{X}_{B_1}$ and  $m'\mathbb{X}_{B_1'}$ are coprime.

By the Laurent phenomenon of cluster variables, we deduce that $\frac{f_Bm\mathbb{X}_{B_1}+f_{B'}m'\mathbb{X}_{B_1'}}{f_{M_k}}$ is a Laurent polynomial in variables $x_1,\cdots, x_n$. It remains to show that $\frac{f_Bm\mathbb{X}_{B_1}+f_{B'}m'\mathbb{X}_{B_1'}}{f_{M_k}}$ is a polynomial satisfying the positive condition.
For each $\epsilon_i$, it follows from the assumption that $f_B(\epsilon_i)>0$ and $f_{B'}(\epsilon_i)>0$. On the other hand, we clearly have $(m\mathbb{X}_{B_1})(\epsilon_i)\geq 0$ and $(m'\mathbb{X}_{B_1'})(\epsilon_i)\geq 0$. Since $m\mathbb{X}_{B_1}$ and  $m'\mathbb{X}_{B_1'}$ are coprime, these two numbers can not be simultaneously zero. Therefore
\[(f_Bm\mathbb{X}_{B_1}+f_{B'}m'\mathbb{X}_{B_1'})(\epsilon_i)>0~\text{for $i=1,\cdots, n$.}
\]
Note that the polynomial $f_{M_k}$ also satisfies the positive condition and we have
\[\frac{f_Bm\mathbb{X}_{B_1}+f_{B'}m'\mathbb{X}_{B_1'}}{f_{M_k}}(\epsilon_i)>0 ~\text{for $i=1,\cdots, n$.}
\]
In particular, $\frac{f_Bm\mathbb{X}_{B_1}+f_{B'}m'\mathbb{X}_{B_1'}}{f_{M_k}}$ is defined for each $\epsilon_i$, which implies that $\frac{f_Bm\mathbb{X}_{B_1}+f_{B'}m'\mathbb{X}_{B_1'}}{f_{M_k}}$ is a polynomial in $x_1,\cdots, x_n$ and hence satisfies the positive condition.

\noindent{\bf Case 2:}  Suppose that $M_k\cong \Sigma T_l$ for some $l$. It is not hard to see that $M_k^*\not\in \add \Sigma T$. By definition, we have $t_{M_k}=1$. According to Lemma~\ref{l:exchange-dimension}, we have 
\begin{eqnarray*}
&&\dim_K\Hom_\cc(T_l, M_k)+\dim_K\Hom_\cc(T_l, M_k^*)\\
&=&\begin{cases}\max\{\dim_K\Hom_\cc(T_l, B),\dim_K\Hom_\cc(T_l, B')\}+1&l\neq 1;\\ \max\{\dim_K\Hom_\cc(T_l, B),\dim_K\Hom_\cc(T_l, B')\}+2&l=1.
\end{cases}
\end{eqnarray*}
On the other hand, for each $j\neq l$,  $T_j$ is compatible with the exchange pair $(M_k,M_k^*)$ by Theorem~\ref{t:exchange compatible}. Therefore, for $j\neq l$, we also obtain
\[\dim_K\Hom_\cc(T_j, M_k)+\dim_K\Hom_\cc(T_j,M_k^*)=\max\{\dim_K\Hom_\cc(T_j, B), \dim_K\Hom_\cc(T_j, B')\}.
\]
Putting all of these together, we have
\begin{eqnarray*}
t_{M_k^*}&=&t_{M_k}t_{M_k^*}\\
&=&x_1^{\frac{\dim_K\Hom_\cc(T_1,M_k)}{2}}\prod_{i=2}^n x_i^{\dim_K\Hom_\cc(T_i, M_k)}x_1^{\frac{\dim_K\Hom_\cc(T_1,M_k^*)}{2}}\prod_{i=2}^n x_i^{\dim_K\Hom_\cc(T_i, M_k^*)}\\
&=&\begin{cases}
x_1x_1^{\max\{\frac{\dim_K\Hom_\cc(T_1,B)}{2},\frac{\dim_K\Hom_\cc(T_1,B')}{2}\}}\prod_{i=2}^nx_i^{\max\{\dim_K\Hom_\cc(T_i, B), \dim_K\Hom_\cc(T_i, B')\}}&l=1\\
x_lx_1^{\max\{\frac{\dim_K\Hom_\cc(T_1,B)}{2},\frac{\dim_K\Hom_\cc(T_1,B')}{2}\}}\prod_{i=2}^nx_i^{\max\{\dim_K\Hom_\cc(T_i, B), \dim_K\Hom_\cc(T_i, B')\}}&l\neq 1
\end{cases}\\
&=&x_l\lcm(t_B, t_{B'}).
\end{eqnarray*}
Note that $\mathbb{X}_{M_k}=x_l$.
By the exchange relation, we have
\[\mathbb{X}_{M_k^*}=\frac{\frac{f_B\mathbb{X}_{B_1}}{t_{B_0}}+\frac{f_{B'}\mathbb{X}_{B'}}{t_{B_0'}}}{x_l}=\frac{f_Bm\mathbb{X}_{B_1}+f_{B'}m'\mathbb{X}_{B_1'}}{t_{M_k^*}},
\]
where $m=\frac{\lcm(t_{B_0}, t_{B_0'})}{t_{B_0}}$ and $m'=\frac{\lcm(t_{B_0},t_{B_0'})}{t_{B_0'}}$. Now similar to the Case $1$, one can show that $f_Bm\mathbb{X}_{B_1}+f_{B'}m'\mathbb{X}_{B_1'}$ is a polynomial satisfying the positive condition and we are done.
\end{proof}
Now we are in a position to state the main result of this subsection.
\begin{theorem}~\label{t:denominator-thm}
 Every cluster variable of $\mathcal{A}_T$ has a $T$-denominator. In particular, for each indecomposable rigid object $M\not\in \add \Sigma T$, we have
 \[\opname{den}(\mathbb{X}_M)=\opname{\underline{rank}} F(M).
 \]
 \end{theorem}
\begin{proof}
Recall that $\Sigma T$ corresponds to the initial cluster under the bijection $\mathbb{X}_?$ and each indecomposable direct summand $\Sigma T_i$ of $\Sigma T$ has a $T$-denominator. On the other hand, each indecomposable rigid object $M$ is a direct summand of a basic maximal rigid object $T'$. Moreover, $T'$ can be obtained from $\Sigma T$ by a finite sequence of mutations. Now the result follows from Lemma~\ref{l:t-denominator}.
\end{proof}

As a direct consequence of the denominator theorem and Theorem~\ref{t:rank-vector-of-tau-rigid}, we obtain the following. 
\begin{corollary}~\label{c:different-denominator}
Let $\ca$ be a cluster algebra of type $\mathrm{C}$. Different cluster variables of $\mathcal{A}$ have different denominator vectors with respect to any given cluster.
\end{corollary}
\subsection{Application to conjectures on denominator vectors}~\label{ss:denominator-thm-app}
In this subsection, we derive certain consequences of the denominator theorem. As a first application, we prove that Conjecture 7.4 of~\cite{FominZelevinsky07} holds true for cluster algebras of type $\mathrm{C}$. Namely,
\begin{theorem}~\label{t:d-conjecture}
Let $\mathcal{A}$ be a cluster algebra of type $\mathrm{C}_n$ with an initial cluster $\mathrm{x}_0=\{x_1,\cdots, x_n\}$. Let $x$ be a non-initial cluster variable with denominator vector $\opname{den}(x)=(d_1,\cdots, d_n)^{\opname{tr}}\in\Z^n$.  Then
\begin{itemize}
\item[(1)] All components $d_i$ are non negative;
\item[(2)] We have $d_i=0$ if and only if there is a cluster containing both $x$ and $x_i$;
\item[(3)] Each component $d_i$ depends only on the cluster variables $x$ and $x_i$.
\end{itemize}
\end{theorem}
\begin{proof}
Let $\cc$ be the cluster tube of rank $n+1$.
Without loss of generality, we may assume that $\mathcal{A}=\mathcal{A}_T$, where $T=\bigoplus\limits_{i=1}^nT_i$ is a basic maximal rigid object of $\cc$ with $l(T_1)=n$. Recall that there is bijection $\mathbb{X}_?$ between the cluster variables of $\mathcal{A}_T$ and the indecomposable rigid objects of $\cc$ by Theorem~\ref{t:BMV-main-theorem}. 
Moreover, the initial cluster variable $x_i$ corresponds to the indecomposable  rigid object $\Sigma T_i$ under the bijection. Denote by $\Gamma=\End_\cc(T)$ the endomorphism algebra of $T$. Let $M$ be an indecomposable rigid object of $\cc$ such that $\mathbb{X}_M=x$. Note that $x$ is non-initial, we have $M\not\in \add \Sigma T$. By Theorem~\ref{t:denominator-thm}, we have $\opname{den}(x)=\opname{\underline{rank}} F(M)$, which implies the statement $(1)$. Furthermore, \[d_i=\begin{cases}\frac{\dim_K\Hom_\cc(T_1, M)}{2}& i=1;\\  \dim_K\Hom_\cc(T_i, M)&\text{else.}\end{cases}\] and we conclude that $d_i$ depends only on $\Sigma T_i$ and $M$. Consequently, $d_i$ depends only on the cluster variables $x$ and $x_i$.

For the second statement, we first assume that there is a component $d_i=0$. Then $(F(M), F(T_i))$ is a $\tau$-rigid pair of the endomorphism algebra $\Gamma$. Each $\tau$-rigid pair can be completed to a support $\tau$-tilting pair~\cite{Adachi-Iyama-Reiten}. In other words, there exists a rigid object $N$ such that $N\oplus M\oplus \Sigma T_i$ is a basic maximal rigid object of $\cc$ by Theorem~\ref{t:LiuXie-ChangZhangZhu}. Now applying the bijection $\mathbb{X}_?$, we deduce that there is  a cluster containing $x$ and $x_i$. Conversely, assume that there is a cluster containing $x$ and $x_i$. Then there is a basic maximal rigid object $N$ such that $M$ and $\Sigma T_i$ belong to $\add N$. Consequently, the component 
\[d_i=\begin{cases}\frac{\dim_K\Hom_\Gamma(F(T_1), F(M))}{2}=\frac{\dim_K\Hom_\cc(T_1, M)}{2}=\frac{\dim_K\Ext^1_{\cc}(\Sigma T_1, M)}{2}=0&i=1;\\
\dim_K\Hom_\Gamma(F(T_i), F(M))=\dim_K\Hom_\cc(T_i, M)=\dim_K\Ext^1_{\cc}(\Sigma T_i, M)=0& \text{else.}
\end{cases}
\]
\end{proof}
\begin{remark}
We remark that the statement $(1)$ of Theorem~\ref{t:d-conjecture} has been known for any cluster algebras of finite type ({\it cf.}~\cite{CCS2, CV}) and has been recently verified for skew-symmetric type in \cite{CaoLi}. We also remark here that this conjecture has been proved for any skew-symmetrizable cluster algebra by using different methods in the arXiv paper \cite{CaoLi2} after the current paper was posted on the arXiv.
\end{remark}

The second application is to prove that the denominator vectors of cluster variables in a cluster form a basis of $\mathbb{Q}^n$.
\begin{theorem}~\label{t:linearly-independence}
Let $\mathcal{A}$ be a cluster algebra of type $\mathrm{C}_n$ with initial cluster $\mathrm{x}_0$.  For any cluster $\mathrm{y}=\{y_1,\cdots, y_n\}$ of $\mathcal{A}$, the denominator vectors $\opname{den}(y_1),\cdots, \opname{den}(y_n)$ are linearly independent over $\mathbb{Q}$.
\end{theorem}
\begin{proof}
We may assume that $\mathcal{A}=\mathcal{A}_T$ for some basic maximal rigid object $T\in \cc$ and the initial cluster $\mathrm{x}_0$ corresponds to the basic maximal rigid object $\Sigma T$ under the bijection $\mathbb{X}_?$. Denote by $\Gamma=\End_\cc(T)$ the endomorphism algebra of $T$.

Note that for an initial cluster variable $x_i$, we have $\opname{den}(x_i)=-e_i\in \mathbb{Z}^n$.
Without loss of generality, we may assume that $y_1,\cdots, y_r$ are precisely the non-initial cluster variables in the cluster $\mathrm{y}$. According to Theorem~\ref{t:d-conjecture} (2), for each $i=1,\cdots, r$, the last $n-r$ components of $\opname{den}(y_i)$ vanish.  Let
$\opname{den}_r(y_i)$ be the vector formed by the first $r$ entries of $\opname{den}(y_i)$. It suffices to show that $\opname{den}_r(y_1),\cdots, \opname{den}_r(y_r)$ are linearly independent over $\mathbb{Q}$. Set $M=M_1\oplus \cdots\oplus M_r$,  where $M_1,\cdots, M_r$ are indecomposable rigid objects in $\cc$ such that $\mathbb{X}_{M_i}=y_i$ for $i=1,\cdots, r$. Denote by $M\oplus \Sigma T_M$  the basic maximal rigid object corresponding to the cluster $\mathrm{y}$. Recall that we have an equivalence $F: \opname{pr} T/\add \Sigma T\to \mod \Gamma$. By Theorem~\ref{t:LiuXie-ChangZhangZhu}, we deduce that $F(M)$ is a $\tau$-tilting module over $A:=\Gamma/\langle e_M\rangle$, where $e_M$ is the idempotent of $\Gamma$ corresponding to the direct summand $T_M$. By Theorem~\ref{t:algbra-maximal-tube}, we know that $A$ is a gentle algebra. 

According to Theorem~\ref{t:denominator-thm}, we have $\opname{den}_r(y_i)=\opname{\underline{rank}} F(M_i)$ for $i=1,\cdots, r$, where $\opname{\underline{rank}} F(M_i)$ is the rank vector of $F(M_i)$ as an $A$-module. Let $G_{F(M)}(A)^{\opname{tr}}$ be the transpose of $\opname{G}$-matrix of $F(M)$ and $D_{F(M)}(A)$ the $\opname{D}$-matrix of $F(M)$. Applying Lemma~\ref{l:g-d-duality}, we obtain
\[G_{F(M)}(A)^{\opname{tr}}D_{F(M)}(A)=C(\End_A(F(M))).
\]
Note that $F(M)$ is a $\tau$-tilting $A$-module, which implies that $\Ext_A^1(F(M),F(M))=0$. Applying Corollary 1.2 of ~\cite{SchroerZimmermann}, we deduce that $\End_A(F(M))$ is also a gentle algebra.
On the other hand, it is not hard to see that \[\End_A(F(M))=\End_\Gamma(F(M))\cong \End_\cc(M)/(\Sigma T)\cong \End_\cc(M\oplus \Sigma T_M)/(\Sigma T),\] where $(\Sigma T)$ is the ideal generated by morphisms factoring through the objects in $\add\Sigma T$. Namely, $\End_A(F(M))$ is a quotient of the endomorphism algebra of the basic maximal rigid object $M\oplus \Sigma T_M$ of $\cc$.
According to Theorem~\ref{t:algbra-maximal-tube} and Remark~\ref{r:cluster-tilted-A}, we conclude that $\End_A(F(M))$ is a gentle algebra without oriented cycles of even length with full relations. Consequnetly, the Cartan matrix $C(\End_A(F(M)))$ of $\End_A(F(M))$ is non-degenerate by Lemma~\ref{l:Cartan-matrix-gentle}. Recall that the $\opname{G}$-matrix $G_{F(M)}(A)$ is invertible over $\mathbb{Z}$. Therefore, the $\opname{D}$-matrix $D_{F(M)}(A)$ is non-degenerate and we conclude that $\opname{\underline{rank}} F(M_1),\cdots, \opname{\underline{rank}} F(M_r)$ are linearly independent over $\mathbb{Q}$. This finishes the proof.
\end{proof}
\begin{remark}~\label{r:linearly-independence-A}
Note that the denominator theorem was also established for cluster algebras of type $\mathrm{A}$ in~\cite{CCS2, BMR}. On the other hand, every cluster-tilted algebra of type $\mathrm{A}$ is a gentle algebra whose quiver is described in Remark~\ref{r:cluster-tilted-A}.  It is clear that the above proof implies the corresponding  statement for cluster algebras of type $\mathrm{A}$.
\end{remark}
\begin{remark}~\label{r:linearly-independence-B}
In ~\cite{ReadingStella}, Reading and Stella established a $D$-matrix duality for denominator vectors of cluster algebras of finite type. In particular, for an arbitrary cluster $\mathrm{y}=\{y_1,\cdots, y_n\}$ of a cluster algebra of type $\mathrm{B}_n$, the set of denominator vectors $\{\opname{den}(y_1),\cdots, \opname{den}(y_n)\}$ coincides with the set of denominator vectors of cluster variables in a cluster for a cluster algebra of type $\mathrm{C}_n$. Therefore we also obtain the linear independence for the denominator vectors for cluster algebras of type $\mathrm{B}$.
\end{remark}

\section{Categorical interpretations for $g$- and $c$-vectors}~\label{s:c-vector}
Let $\cc$ be the cluster tube of rank $n+1$ and $T=\bigoplus\limits_{i=1}^nT_i$ a basic maximal rigid object of $\cc$ with $l(T_1)=n$. Let $B_T$ be the skew-symmetrizable matrix associated to $T$ and  $\ca_{T,pr}$ the cluster algebra with principal coefficients associated to the matrix $B_T$. By Theorem~\ref{t:BMV-main-theorem}, there is a bijection between indecomposable rigid objects of $\cc$ and the cluster variables of $\ca_{T,pr}$. The aim of this section is to give a categorical interpretation for $g$-vectors and $c$-vectors of $\ca_{T,pr}$.

\subsection{$g$-vector  as index}

Recall that $\opname{pr} T$ is the full subcategory of $\cc$ consisting of objects which are finitely presented by $T$. Moreover, each indecomposable rigid object belongs to $\opname{pr} T$.
Let $\go(\add T)$ be the split Grothendieck group of $\add T$.
For each $X\in \opname{pr}T$, we have a triangle
\[T_1^X\to T_0^X\to X\to \Sigma T_1^X
\]
with $T_1^X, T_0^X\in \add T$.
The {\it index} $\opname{ind}_T(X)$ of $X$ with respect to $T$ is defined to be
\[\opname{ind}_T(X):=[T_0^X]-[T_1^X]\in \go(\add T),
\]
where $[*]$ stands for the image of the object $*$ in the Grothendieck group $\go(\add T)$. We define the {\it $g$-vector} $g_T(X)=(g_1,\cdots, g_n)^{\opname{tr}}$ of $X$ with respect to $T$ as the coordinate vector of $\opname{ind}_T(X)$ with respect to the standard basis $[T_1], \cdots, [T_{n}]$ of $\go(\add T)$.

Let $\Gamma=\End_\cc(T)$ be the endomorphism algebra of $T$.
 Recall that we have an equivalence
\[F:\opname{pr} T/\add \Sigma T\to \mod \Gamma
\] 
and the functor $F$ induces a bijection between the basic maximal rigid objects of $\cc$ and the basic support $\tau$-tilting $\Gamma$-modules ({\it cf.}~Theorem~\ref{t:LiuXie-ChangZhangZhu}).
 The g-vector of a rigid object in $\cc$ has close relation to the g-vector of a $\tau$-rigid $\Gamma$-module.
 Indeed, for an indecomposable rigid object $X\not\in \add \Sigma T$, $F(X)$ is an indecomposable $\tau$-rigid $\Gamma$-module and 
 one can show that $g_T(X)$ coincides with the g-vector $g(F(X))$ of the $\tau$-rigid $\Gamma$-module $F(X)$.
 Following ~\cite{DehyKeller08}, one can show that different rigid objects have different indices and $g$-vectors. Moreover, if $T'$ is another basic maximal rigid object of $\cc$, then the indices of the direct summands of $T'$ with respect to $T$ form a basis of $\go(\add T)$. 

Now fix a vertex $t_0\in \mathbb{T}_n$. We associate to $t_0$ the $2n\times n$ matrix $\tilde{B}_T=\left(\begin{array}{c}B_T\\ E_n\end{array}\right)$ to get the cluster algebra with principal coefficents $\ca_{T,pr}:=\ca(\tilde{B}_T)$.  To each vertex $t\in \mathbb{T}_n$, we associate a basic maximal rigid object $T_t=\bigoplus\limits_{i=1}^nT_{t,i}$ such that $T_{t_0}=T$ and $T_{t'}=\mu_k(T_t)$ whenever $t\frac{~~k~~}{}t'$ is an edge of $\mathbb{T}_n$. Let
\[T_{t,k}^*\to U_{T_{t,k}, T_{t}\backslash T_{t,k}}\to T_{t,k}\to \Sigma T_{t,k}^*~ \text{and}~T_{t,k}\to U'_{T_{t,k}, T_{t}\backslash T_{t,k}}\to {T_{t,k}^*}\to \Sigma T_{t,k}
\]
be the exchange triangles associated to $T_{t,k}$ and $T_{t,k}^*$. Following ~\cite{DehyKeller08}, we define two linear maps  $\phi^{\pm}_{t,t'}:\go(\add T_t)\to \go(\add T_{t'})$ which both send each $[N]$ for an indecomposable object $N$ belonging to both $\add T_{t}$ and $\add T_{t'}$ to itself and such that
\[\phi^+_{t,t'}([T_{t,k}])=[ U_{T_{t,k}, T_{t}\backslash T_{t,k}}]-[T_{t,k}^*]~~\text{and}~~\phi^{-}_{t,t'}([T_{t,k}])=[U'_{T_{t,k}, T_{t}\backslash T_{t,k}}]-[T_{t,k}^*].
\]
The following result is due to Dehy and Keller~\cite{DehyKeller08}.
\begin{theorem}~\label{t:g-vector and index}
\begin{itemize}
\item[(1)]For any rigid object $X$ in $\cc$, we have
\[\ind_{T_{t'}}(X)=\begin{cases}\phi^+_{t,t'}(\ind_{T_t}(X))& ~\text{if}~ [\ind_{T_t}(X):T_{t,k}]\geq 0;\\ \phi^-_{t,t'}(\ind_{T_t}(X))& ~\text{if}~ [\ind_{T_t}(X):T_{t,k}]\leq 0,\end{cases}
\]
where $ [\ind_{T_t}(X):T_{t,k}]$ is the coefficient of $[T_{t,k}]$ in $\ind_{T_t}(X)$ with respect to the basis $[T_{t,1}], \cdots, [T_{t,n}]$.
\item[(2)] For each vertex $t\in \mathbb{T}_n$ and $1\leq l\leq n$, we have $g_{l,t}^{\tilde{B}_T, t_0}=g_T(T_{t,l})$.
\end{itemize}

\end{theorem}
\begin{proof}
The same proof of Section $3$ and $4$ in ~\cite{DehyKeller08} applied. We sketch a proof for the second statement to emphasize why the indices categorify the $g$-vectors of cluster algebra of type $\mathrm{C}$ but not of type $\mathrm{B}$.   It suffices to prove that the $g$-vectors $g_T(T_{t,l})$ satisfy the same recurrence formula as in Proposition~\ref{p:property-c-type} $(4)$.  Let $t_1\frac{~~k~~}{}t_2$ be two vertices of $\mathbb{T}_n$ and \[T_{t_1,k}^*\to U_{T_{t_1,k}, T_{t_1}\backslash T_{t_1,k}}\to T_{t_1,k}\to \Sigma T_{t_1,k}^*~ \text{and}~T_{t_1,k}\to U'_{T_{t_1,k}, T_{t_1}\backslash T_{t_1,k}}\to {T_{t_1,k}^*}\to \Sigma T_{t_1,k}
\]
 the exchange triangles associated to $T_{t_1,k}$. Denote by $B_{T_{t_1}}=(b_{ij}^{t_1})$ the matrix associated to $T_{t_1}$. Let $g_{T_{t_1}}(T_{t,l})=(g_1,\cdots, g_n)$ and $ g_{T_{t_2}}(T_{t,l})=(g_1', \cdots, g_n')$ be the $g$-vectors of $T_{t,l}$ with respect to $T_{t_1}$ and $T_{t_2}$ respectively. We have to show that
 \[g_j'=\begin{cases}-g_k& ~\text{if}~ j=k;\\ g_j+[b_{jk}^{t_1}]_{+}g_k-b_{jk}^{t_1}\min(g_k,0)& \text{else}.\end{cases}
 \]
Recall that $g_i=[\ind_{T_{t_1}}(T_{t,l}):T_{t_1,k}]$. Assume that $g_k\geq 0$.  By the first statement, we have
\begin{eqnarray*}
\ind_{T_{t_2}}(T_{t,l})&=&\phi^+_{t_1,t_2}(g_1[T_{t_1,1}]+\cdots+g_n[T_{t_1,n}])\\
&=&\sum_{j\neq k}g_j[T_{t_1,j}]+g_k\phi^+_{t_1,t_2}(T_{t_1,k})\\
&=&\sum_{j\neq k}g_j[T_{t_1,j}]+g_k([U_{T_{t_1,k}, T_{t_1}\backslash T_{t_1,k}}]-[T_{t_1,k}^*]).
\end{eqnarray*}
Note that $T_{t_1,k}^*=T_{t_2,k}$ and $U_{T_{t_1,k}, T_{t_1}\backslash T_{t_1,k}}$ dose not admit $T_{t_1,k}^*$ as a direct summand, we deduce that $g_k'=-g_k$. Now for $j\neq k$, by the definition of the matrix $B_{T_{t_1}}$, we have
\[b^{t_1}_{jk}=\alpha_{U_{T_{t_1,k}, T_{t_1}\backslash T_{t_1,k}}}T_{t_1,j}-\alpha_{U'_{T_{t_1,k}, T_{t_1}\backslash T_{t_1,k}}}T_{t_1,j}.
\]
In particular, we have
\[[U_{T_{t_1,k}, T_{t_1}\backslash T_{t_1,k}}:T_{t_1,j}]=\begin{cases}b_{jk}^{t_1}~&~\text{if}~b_{jk}^{t_1}\geq 0;\\ 0&~\text{otherwise}.\end{cases}
\]
It follows that $g_j'=g_j+g_k[b_{jk}^{t_1}]_{+}$ for $j\neq k$. Similarly, if $g_k\leq 0$, one can show that
\[g_j'=\begin{cases}-g_k&~\text{if}~j=k;\\ g_j+g_k[-b_{jk}^{t_1}]_+&~\text{if}~ j\neq k.\end{cases}
\]
This completes the proof.

\end{proof}

\subsection{$c$-vector as rank vector}
 Let $\opname{cv}(\ca_{T,pr})$ be the set of $c$-vectors of $\ca_{T,pr}$ and $\opname{cv}^+(\ca_{T,pr})$ the set of positive $c$-vectors, then \[\opname{cv}(\ca_{T,pr})=\opname{cv}^+(\ca_{T,pr})\cup -\opname{cv}^+(\ca_{T,pr}).\] The following result  interprets positive $c$-vectors as rank vectors of indecomposable $\tau$-rigid modules.
\begin{theorem}~\label{t:positive-c-vectors}
Let $T=\bigoplus\limits_{i=1}^nT_i$ be a basic maximal rigid object in $\cc$ and $\Gamma$ the endomorphism algebra of $T$.  We have
\[\opname{cv}^+(\ca_{T,pr})=\{\rankv~M~|~\text{$M$ is an indecomposable $\tau$-rigid $\Gamma$-module}\}.
\]
\end{theorem}
\begin{proof}
As before, the cluster pattern for the cluster algebra $\ca_{T,pr}$ is given by assigning the initial exchange matrix $\tilde{B}_T$ to the root vertex $t_0$.  In particular, for each vertex $t\in \mathbb{T}_n$, we have the $G$-matrix $G_t$ and the $C$-matrix $C_t$ for the cluster algebra $\ca_{T,pr}$. To each vertex $t\in \mathbb{T}_n$, we also associate a basic maximal rigid object $T_t=\bigoplus\limits_{i=1}^nT_{t,i}$ such that $T_{t_0}=T$ and $T_{t'}=\mu_k(T_t)$ whenever $t\frac{~~k~~}{}t'$ is an edge of $\mathbb{T}_n$. Note that for each $t\in \mathbb{T}_n$, $\Hom_{\cc}(T, T_t)$ is a support $\tau$-tilting $\Gamma$-module. We denote by $G_t(\Gamma)$ and $C_t(\Gamma)$ the corresponding $G$-matrix and $C$-matrix  associated to the support $\tau$-tilting $\Gamma$-module $\Hom_{\cc}(T, T_t)$ with respect to the decomposition induced by $T_t$.  Now Theorem~\ref{t:g-vector and index} implies that for any $t\in \mathbb{T}_n$, we have $G_t=G_t(\Gamma)$. Hence by Proposition ~\ref{p:property-c-type} $(3)$, we have 
\[C_t=D^{-1}C_t(\Gamma)D,\]where $D\in M_n(\Z)$ is the skew-symmetrizer of $\tilde{B}_T$.

Now assume that $l(T_1)=n$. In other words, $D=\opname{diag}\{2, 1,\cdots, 1\}$.
Let $M$ be an indecomposable $\tau$-rigid $\Gamma$-module with 
\[\dimv M=(m_1,m_2,\cdots, m_n)^{\opname{tr}},
\]
we then have
\[\rankv M=(\frac{m_1}{2}, m_2,\cdots, m_n)^{\opname{tr}}.
\]
By Proposition~\ref{p:property-c-type} $(2)$ and Theorem~\ref{t:rank-vector-of-tau-rigid}, it suffices to  show that for each indecomposable $\tau$-rigid $\Gamma$-module $M$, the  rank vector $\rankv~M$ is a positive $c$-vector of the cluster algebra $\ca_{T,pr}$.

Let $\tilde{M}$ be the indecomposable preimage of $M$ in $\opname{pr} T$. We may choose a maximal rigid object in $\cc$, say $T_{t'}:=\tilde{M}\oplus \tilde{N}$, such that $\Hom_{\cc}(\tilde{N},\tilde{M})=0$. Note that $\Hom_{\cc}(T, \tilde{M}\oplus \tilde{N})$ is a support $\tau$-tilting $\Gamma$-module. Let $Q_{M}$ be the corresponding $2$-term silting object in $\per \Gamma$. If $l(\tilde{M})<n$, then $\dimv M$ is a $c$-vector of $\Gamma$ by Proposition~\ref{p:criterion}. Recall that by Lemma \ref{l:basic-property-cluster-tube} (3), in the following exchange triangle of a basic maximal rigid object $X=X_1\oplus\cdots\oplus X_n$
\[X_k\to Y\to X_k^*\to \Sigma X_k,
\]
  the indecomposable object $X_k$ has length $n$ if and only if $X_k^*$ has length $n$. Consequently, if $l(\tilde{M})<n$, then $\dimv M$ is not the first column of $C_{t'}(\Gamma)$.  By $C_{t'}=D^{-1}C_{t'}(\Gamma)D$, we deduce that $\rankv M$ is a column vector of $C_{t'}$ and hence a positive $c$-vector of $\ca_{T,pr}$.
 Now if $l(\tilde{M})=n$,  by Lemma~\ref{l:morphism of quasi length n}, we have 
 \[\Hom_{\der^b(\mod \Gamma)}(Q_M, \Sigma^iM)=\begin{cases}K\oplus K,& i=0;
 \\ 0,& \text{else.}\end{cases}\]
 In other words, $\frac{1}{2}\dimv M$ is a $c$-vector of $\Gamma$. Moreover, $\frac{1}{2}\dimv M$ is the first column vector of $C_{t'}(\Gamma)$. Again the equality $C_{t'}=D^{-1}C_{t'}(\Gamma)D$ implies that $\rankv M$ is the first column vector of $C_{t'}$ and hence a positive $c$-vector of $\ca_{T,pr}$. This finishes the proof.
\end{proof}

\subsection*{Acknowledgments}
We are very grateful to the anonymous referee for significant comments and
corrections they proposed.


\begin{thebibliography}{99}

\bibitem{Adachi-Iyama-Reiten}
T. Adachi, O. Iyama and I. Reiten, \emph{$\tau$-tilting theory}, Compos. Math. \textbf{150}(2014), no. 3, 415-452.



\bibitem{AD}I. Assem, G. Dupont, \emph{ Modules over cluster-tilted algebras determined by their dimension vectors},  Comm. Algebra \textbf{41}(12)(2013), 4711-4721.


\bibitem{AS87}I. Assem,  A. Skowro{\'n}ski, \emph{Iterated tilted algebras of type $\tilde{A_n}$},   Math. Z., \textbf{195}(2)(1987), 269--290.

\bibitem{Barot-Kussin-Lenzing}
M. Barot, D. Kussin and H. Lenzing, \emph{The Grothendieck group of a cluster category}, J. Pure Appl. Algebra \textbf{212}(1)(2008), 33-46.

\bibitem{BMRRT}
A. B. Buan, R. Marsh, M. Reineke, I. Reiten and G. Todorov, \emph{Tilting theory and cluster combinatorics}, Adv. Math. \textbf{204}(2)(2006), 572-618.

\bibitem{BMR}
A. Buan, R. Marsh and I. Reiten, \emph{Denominators of cluster variables}, J. Lond. Math. Soc. (2) \textbf{79}(3)(2009), 589-611.



\bibitem{BuanMarshVatne}
A. B. Buan, R. Marsh and D. F. Vatne, \emph{Cluster structures from $2$-Calabi-Yau categories with loops}, Math. Z. \textbf{256}(4)(2010), 951-970.

\bibitem{BuanVatne}
A. B. Buan and D. F. Vatne, \emph{Derived equivalence classification for cluster-tilted algebras of type An},  J. Algebra \textbf{319}(7)(2008), 2723-2738.


\bibitem{ButlerRingel}
M.C.R. Butler and C.M. Ringel, \emph{Auslander-Reiten sequences with few middle terms and applications to string algebras}, Comm.  Algebra \textbf{15}(1987), 145-179.

\bibitem{CalderoChapoton}
P. Caldero and F. Chapoton, \emph{Cluster algebras as Hall algebras of quiver representations}, Comment. Math. Helv. \textbf{81}(2006), no. 3, 595-616.

\bibitem{CCS2}
P. Caldero, F. Chapoton and R. Schiffler, \emph{Quiver with relations and cluster tilted algebras}, Algebra Represent. Theory \textbf{9}(2006), no. 4, 359-376.

\bibitem{CalderoKeller}
P. Caldero and B. Keller, \emph{From triangulated categories to
cluster algebras II}, Ann. Sci. Ecole Norm. Sup. 4eme serie,
\textbf{39}(2006), 983-1009.


\bibitem{CaoLi}
P. Cao and F. Li, \emph{Study on cluster algebras via abstract pattern and two conjectures on $d$-vectors and $g$-vectors}, arXiv:1708.08240.

\bibitem{CaoLi2}
P. Cao and F. Li, \emph{The enough $g$-pairs property and denominator vectors of cluster algebras}, arXiv:1803.05281.

\bibitem{CV}
C. Ceballos and V. Pilaud, \emph{Denominator vectors and compatibility degrees in cluster algebras of finite type}, Trans. Amer. Math. Soc. \textbf{367}(2)(2015), 1421-1439.


\bibitem{ChangZhangZhu}
W. Chang, J. Zhang and B. Zhu, \emph{On support $\tau$-tilting modules over endomorphism algebras of rigid objects}, Acta Math. Sinica, English Series \textbf{31}(9)(2015), 1508-1516.


\bibitem{DehyKeller08}
R. Dehy and B. Keller, \emph{On the combinatorics of rigid objects in $2$-{C}alabi-{Y}au categories}, Int. Math. Res. Notices
  \textbf{2008}(2008), rnn029-17.
  
\bibitem{Demonet10}
L. Demonet, \emph{Mutations of group species with potentials and their representations. Applications to cluster algebras}, arXiv:1003.5078.

\bibitem{Demonet}
L. Demonet, \emph{Categorification of skew-symmetrizable cluster algebras}, Algebra Represent. Theory \textbf{14}(6)(2011), 1087-1162.


\bibitem{DWZ10}
H. Derksen, J. Weyman and A. Zelevinsky, \emph{Quivers with potentials and their representations II: Application to cluster algebras}, J. Amer. Math. Soc. \textbf{23}(2010), 749-790.


\bibitem{FominZelevinsky02}
S. Fomin and A. Zelevinsky, \emph{Cluster algebras I: Foundations}, J. Amer. Math. Soc. \textbf{15}(2)(2002), 497-529.

\bibitem{FominZelevinsky03}
S. Fomin and A. Zelevinsky, \emph{Cluster algebras II: finite type classification}, Invent. Math. \textbf{154}(2003), 63-121.

\bibitem{FominZelevinsky03b}
S. Fomin and A. Zelevinsky, \emph{Y-systems and generalized associahedra}, Ann. Math. \textbf{158}(2003), 977-1018.

\bibitem{FominZelevinsky04}
S. Fomin and A. Zelevinsky, \emph{Cluster algebras: notes for the CDM-03 conference}, in \emph{Current Developments in Mathematics} (International Press, Boston, MA, 2004).

\bibitem{FominZelevinsky07}
S. Fomin and A. Zelevinsky, \emph{Cluster algebras IV: Coefficients}, Compos. Math. \textbf{143}(2007), 112-164.


\bibitem{Fu17}
C. Fu, \emph{$c$-vectors via $\tau$-tilting theory}, J. Algebra \textbf{473}(2017), 194-220.

\bibitem{FuGeng}
C. Fu and S. Geng, \emph{On indecomposable $\tau$-rigid modules for cluster-tilted algebras of tame type}, J. Algebra \textbf{531}(2019), no. 5, 1239-1260.

\bibitem{FuGengLiu}
C. Fu, S. Geng and P. Liu, \emph{Cluster algebras arising from cluster tubes II: The Caldero-Chapoton map}, J. Algebra \textbf{544}(2020), 228-261.



\bibitem{GLS14}
C. Geiss, B. Leclerc and J. Schr\"{o}er,
\emph{Quivers with relations for symmetrizable Cartan matrices I: Foundations}, Invent. Math. \textbf{209}(2017), 61-158.

\bibitem{GLS17}
C. Geiss, B. Leclerc and J. Schr\"{o}er,
\emph{Quivers with relations for symmetrizable Cartan matrices V:  Caldero-Chapoton formulas}, Proc. Lond. Math. Soc. (3) \textbf{117}(2018), no. 1, 125-148.


\bibitem{GHKK}
M. Gross, P. Hacking, S. Keel and M. Kontsevich, \emph{Canonical bases for cluster algebras}, J. Amer. Math. Soc. \textbf{31}(2018), no. 2, 497-608.

\bibitem{Holm05}
T. Holm, \emph{Cartan determinants for gentle algebras}, Arch. Math. \textbf{85}(2005), 233-239.

\bibitem{Kac}
V. Kac, \emph{Infinite dimensional Lie algebras}, Third edition, Cambridge University Press, Cambridge, 1990.

\bibitem{Keller05}
B. Keller, \emph{On triangulated orbit categories}, Doc. Math. \textbf{10}(2005), 551-581.


\bibitem{Keller11}
B. Keller, \emph{Cluster algebras and cluster categories}, Bull. Iranian Math. Soc. \textbf{37}(2)(2011), 187-234.



\bibitem{KR}
B. Keller and I. Reiten, \emph{Cluster-tilted algebras are Gorenstein and stably Calabi-Yau}, Adv. Math. \textbf{211}(1)(2007), 123-151.


\bibitem{LiuXie}
 P. Liu, Y. Xie,  \emph{On the relation between maximal rigid objects and $\tau$-tilting modules},
Colloq. Math. \textbf{142}(2)(2016), 169-178.
\bibitem{Nagao13}
K. Nagao, \emph{Donaldson-Thomas theory and cluster algebras}, Duke Math. J. \textbf{162}(2013), no. 7,  1313-1367.

\bibitem{MRZ03}
R. Marsh, M. Reineke and A. Zelevinsky, \emph{Generalized associahedra via quiver representations}, Trans. Amer. Math. Soc. \textbf{355}(2003), no. 10, 4171-4186 (electronic).

\bibitem{Chavez13}
A. N\'{a}jera Ch\'{a}vez, \emph{On the $c$-vectors of an acyclic cluster algebra}, Int. Math. Res. Notices (2015), no. 6, 1590-1600.
\bibitem{Chavez14}
A. N\'{a}jera Ch\'{a}vez, \emph{$c$-vectors and dimension vectors for cluster-finite quivers},  Bull. Lond. Math. Soc. \textbf{45}(2013), 1259-1266.
\bibitem{NS14}
T. Nakanishi and S. Stella, \emph{Diagrammatic description of $c$-vectors and $d$-vectors of cluster algebras of finite type}, Electron. J. Combin. \textbf{21}(2014), 107 pages.

\bibitem{NZ12}
T. Nakanishi and A. Zelevinsky, \emph{On tropical dualities in cluster algebras}, Contemp. Math. \textbf{565}(2012), 217-226.

\bibitem{Palu}
Y. Palu, \emph{Cluster characters for $2$-Calabi-Yau triangulated categories}, Ann. Inst. Fourier, Grenoble \textbf{56}(6)(2008), 2221-2248.

\bibitem{Plamondon11}
P. Plamondon, \emph{Cluster algebras via cluster categories with infinite-dimensional morphism spaces}, Compos. Math. \textbf{147}(2011),1921-1954.


\bibitem{ReadingStella}
N. Reading and S. Stella, \emph{Initial-seed recursions and dualities for d-vectors},  Pacific J. Math. \textbf{293}(2018), no. 1, 179-206.

\bibitem{RupelStella}
D. Rupel and S. Stella, \emph{Some consequences of categorification}, arXiv:1712.08478. 

\bibitem{SchroerZimmermann}
J. Schr\"{o}er and A. Zimmermann, \emph{Stable endomorphism algebras of modules over special biserial algebras}, Math. Z. \textbf{244}(2003), 515-530.

\bibitem{SZ}
P. Sherman and A. Zelevinsky, \emph{Positivity and canonical bases in rank $2$ cluster algebras of finite and affine types}, Moscow Math. J. \textbf{4}(2004), no. 4, 947-974.

\bibitem{Vatne11}
D. F. Vatne, \emph{Endomorphism rings of maximal rigid objects in cluster tubes}, Colloq. Math.\textbf{123}(2011), 63-93.

\bibitem{Yang12}
D. Yang, \emph{Endomorphism algebras of maximal rigid objects in cluster tubes}, Comm. Algebra \textbf{40}(2012), 4347-4371.

\bibitem{ZhouZhu}
Y. Zhou and B. Zhu, \emph{Maximal rigid subcategories in 2-Calabi-Yau triangulated
categories},  J. Algebra \textbf{348}(2011), 49-60.
\bibitem{ZhouZhu14}
Y. Zhou and B. Zhu, \emph{Cluster algebras arising from cluster tubes}, J. Lond. Math. Soc. \textbf{89}(3)(2014), 703-723.
\bibitem{Zhu07}
B. Zhu, \emph{BGP-reflection functors and cluster combinatorics},  J. Pure Appl. Algebra \textbf{209}(2007), 497-506.
\end{thebibliography}
\end{document}